\documentclass[11pt,letterpaper,reqno]{amsart}
\usepackage[margin=1in]{geometry}
\usepackage{color}
\usepackage[latin1]{inputenc}
\usepackage{amsmath,amsfonts,amsthm,epsfig,graphicx,amssymb}
\usepackage{esint}
\usepackage{caption}
\usepackage{stmaryrd}
\usepackage[normalem]{ulem}
\usepackage[colorlinks=true, linkcolor=blue, citecolor=blue]{hyperref}
\usepackage{overpic}

\newcommand{\N}{\mathbb{N}}

\newcommand{\R}{\mathbb{R}}

\newcommand{\e}{\varepsilon}

\newcommand{\pa}{\partial}

\newcommand\restr[2]{{
		\left.\kern-\nulldelimiterspace 
		#1 
		\right|_{#2} 
}}

\theoremstyle{plain}
\newtheorem{theorem}{Theorem}[section]
\newtheorem{lemma}[theorem]{Lemma}
\newtheorem{corollary}[theorem]{Corollary}
\newtheorem{proposition}[theorem]{Proposition}
\newtheorem*{theorem*}{Theorem}
\newtheorem*{corollary*}{Corollary}

\theoremstyle{definition}

\newtheorem{remark}[theorem]{Remark}

\newtheorem*{notation*}{Notation}

\numberwithin{equation}{section}
\numberwithin{figure}{section}

\allowdisplaybreaks

\usepackage{graphicx}

\title{Convergence of semilinear parabolic flows with general initial data}
\author{Daniel Restrepo}
\address{Department of Mathematics, Johns Hopkins University, Baltimore, MD, United States of America}
\email{drestre1@jh.edu}

\begin{document}

\begin{abstract} {\rm 
We analyze the long-time behavior of solutions to semilinear parabolic equations in the Euclidean space arising as gradient flows of an energy functional $J$. We prove that, for general initial data --including those with non-compact support-- the flow converges to a unique ground state of $J$. The argument relies on a sharp stability estimate for almost critical points of \(J\), providing a flexible framework to prove convergence of gradient flows associated with constrained minimization problems in \(\mathbb{R}^n\). As an application we strengthen the convergence results of \cite{cortazar1999uniqueness, feireisl1997threshold}. }
	\end{abstract}
	
	\maketitle

    \section{Introduction}
\subsection{Overview} This paper is concerned with the convergence as $t \to \infty$ of non-negative solutions $u = u(t,x)$ to the Cauchy problem
\begin{equation}\label{eq generalflow}
\begin{cases}
\partial_t u=   \Delta u - f(u) \quad \text{in} \,\, (0,\infty)\times \R^n, \\
 u(t,x) \to 0, \hspace{1.5cm}  \text{as $|x| \to \infty$},\\
u(.,0)=u_0.
\end{cases}
\end{equation}
Here $f$ is a function that meets some standard growth and regularity conditions so that problem \eqref{eq generalflow} can be recast as the $L^2$-gradient flow associated with the energy functional
\begin{equation}\label{eq energy}
    J(v) = \int_{\R^n} \frac{1}{2} |\nabla v|^2 + F(v),
\end{equation}
where $F(v) = \int_0^v f(s)ds$. In this case, one has that $J$ is decreasing along solutions and, furthermore, 
\begin{equation}\label{eq derentropy}
    \frac{d}{dt} J(u(t,\cdot)) = - \int_{\R^n} (\partial_t u)^2
\end{equation}
 natural question is whether solutions of \eqref{eq generalflow} converge, as $t \to \infty$, to a non-negative solution (ground state) of the associated stationary problem
\begin{equation}\label{eq gs}
\begin{cases}
\Delta \xi = f(\xi) \quad  \R^n, \\
 \xi(x) \to 0, \quad  \text{as $|x| \to \infty$}.\\
\end{cases}
\end{equation}
A key difficulty in addressing this question in $\R^n$ arises from the translation invariance of $J$, which implies lack of compactness of the set of stationary solutions and allows for bubbling phenomena along the flow. To illustrate this issue, assume that
\[
\sup_{t>1}\|u(t)\|_{W^{1,2}(\R^n)} < \infty,
\qquad
\|\partial_t u(t)\|_{L^2(\R^n)} \to 0
\quad \text{as } t\to\infty.
\]
By local weak compactness in $W^{1,2}(\R^n)$, there exist times $t_k \to \infty$ such that $u(t_k,\cdot)$ converges locally in $L^2$ to some solution $\xi_0$ of \eqref{eq gs}. However, this result is unsatisfactory: it only yields subsequential convergence and does not exclude loss of $L^2$--mass at infinity. Under suitable assumptions on $f$ (see \cite{lions1988positive, feireisl1997long, feireisl1997threshold}), concentration--compactness theory refines this description. One obtains that $u_k$ converges to a superposition of solutions $\{\xi_i\}_{i=1}^M$ of \eqref{eq gs}, i.e.,
\begin{eqnarray}\notag
   &&\lim_{k\to \infty} \Big\Vert u_k  -\sum_{i=1}^M \xi_i(\cdot - x_k^i) \Big \Vert_{L^2(\R^n)} =0,\\ \label{eq bubbling}
   &&\lim_{k\to \infty} |x^i_k -x^j_k| \to \infty \qquad \mbox{for $i\neq j$, $i,j =1, \cdots, M.$}
\end{eqnarray}
Statements like \eqref{eq bubbling} are known in the literature as \emph{bubbling phenomena} and reflects the lack of compactness of Palais--Smale sequences in certain variational problems in $\R^n$. For instance, in the context of functional inequalities, it is known that almost critical points to the Sobolev embedding in the Euclidean space are not necessarily close to Talenti bubbles (i.e., the class of functions that saturates Sobolev inequality) but to supperpositions of them, see \cite{struwe1984global}. In the parabolic setting, \eqref{eq bubbling} arises by viewing $\partial_t u(t_k)$ as a vanishing perturbation of the elliptic equation $\nabla J(u) = \partial_t u$.

The decomposition \eqref{eq bubbling} is particularly useful when \emph{a priori} information on $u_0$ prevents bubbling (e.g., compact support or energy below a threshold; see \cite{feireisl1997long, feireisl1997threshold, bonforte2024asymptotic}). Whether bubbling may occur for general initial data was posed as an open problem in \cite{feireisl1997long}. The following theorem provides a partial answer.

\begin{theorem}\label{thm1}
Let $n\geq 3$ and \begin{equation}\label{eq tildef}
f(t) = a_0t -\sum_{l=1}^La_i t^{p_l},  
\end{equation}
where $a_i\geq 0$, $a_0>0$, $\sum_{i=1}^La_i>0$, and $1<r<p_i < \frac{n}{n-2}$. Let $u_0 \in W^{1,2}(\R^n)$ with $u_0 \geq 0$ almost everywhere in $\R^n$ and let $u$ be a solution to \eqref{eq generalflow}. Then, $u$ is the unique solution to \eqref{eq generalflow}, $u(t,x)>0$ for $t>0$, and either \begin{equation}\label{eq zerop}
 \lim_{t\to \infty}J(u(t, \cdot))=0, \qquad  \lim_{t\to \infty}\Vert u( t, \cdot)\Vert_{W^{2,2}(\R^n)\cap C^2(\R^n)} \to 0 
\end{equation}
or
\begin{equation}\label{eq nonzero}
   J(u(t, \cdot))-J(\xi) \leq C e^{-\frac{t}{C}}, \qquad \Vert u(t, \cdot)-\tau_{x^*}[\xi]\Vert_{W^{2,2}(\R^n)\cap C^2(\R^n)} \leq C e^{-\frac{t}{C}},
\end{equation}
$\xi$ is the unique radial positive solution of \eqref{eq gs} and $x^* \in \R^n$, and where $C>0$ depends on $n$, $u_0$, and $f$.
\end{theorem}

\begin{remark}\label{remark compactly}
There is a substantial literature on the asymptotic behavior of solutions to semilinear heat equations in bounded domains; see \cite{busca2002convergence} and references therein. Many of these techniques do not extend directly to unbounded domains because of the lack of compactness. In the case of $\R^n$, additional difficulties arise due to the translation invariance of the energy functional $J$.

 A standard approach to regain compactness is to assume that $u_0$ is compactly supported. Combined with the maximum principle or the moving planes method, this ensures uniform decay of solutions along the flow \cite{feireisl1997long, cortazar1999uniqueness, busca2002convergence, feireisl1997threshold, foldes2011convergence}. This assumption, together with the decomposition \eqref{eq bubbling}, reduces the problem to studying the asymptotic behavior around a single stationary solution. Under these conditions, uniqueness of the limiting profile and convergence rates follow, provided that the stationary problem \eqref{eq gs} has a unique solution $\xi$ (up to translations) and that the kernel of $D^2J(\xi)$ is generated by $\{\partial_i \xi\}_{i=1}^n$. Similar results were obtained in \cite{cortazar1999uniqueness} using different methods but still assuming compactly supported initial data.
 
Remarkably, under the minimal assumptions $f \in C^1$, $f(0)=0$, and $f'(0)>0$, full convergence to a unique stationary profile was shown in \cite{busca2002convergence, foldes2011convergence}, again assuming compact support of the initial datum. These results do not require uniqueness or non-degeneracy of stationary solutions, but they do not provide rates of convergence, not even algebraic.
\end{remark}

\begin{remark}\label{remark hypf}(On the structure of $f$)
   Although Theorem \ref{thm1} assumes the specific form \eqref{eq tildef}, our methods apply to a broader class of reaction terms. In general, we work in the framework where the hypotheses $f \in C^{1,\beta}_{\rm loc}([0,\infty))$, $f(0)=0$, and $f'(0)>0$ are satisfied. Under these general assumptions, ground states are radially symmetric (up to translations) and have a precise asymptotic profile at infinity; see Proposition \ref{prop radialuniqueness}.

   In this setting, convergence of global solutions starting from compactly supported initial data was established in \cite{busca2002convergence} (see Remark \ref{remark compactly}). Our approach for non-compactly supported initial data is variational: convergence is proved by showing that the solution approaches a unique ground state minimizing a constrained variational problem. For our choice of $f$, this follows from Berestycki and Lions' theory \cite{berestycki1980existence}, which covers a broader class of nonlinearities. 

   We also require: (i)  uniqueness of bounded ground states up to translations (see Section \ref{sec background}); and (ii)  non-degeneracy of $D^2J$ at ground states (see Proposition \ref{thm nondegneracy}). These two properties form the variational framework underpinning our methods.  

   Finally, we impose two additional conditions on $f$:
\begin{equation}\label{eq KPP} \tag{H1}
t f'(0) \ge f(t), \quad t \in [0, \max \xi],
\end{equation}
where $\xi$ is the unique radial ground state; and
\begin{equation}\label{eq concat0}\tag{H2}
f \text{ is concave in } (0,\delta) \text{ for some } \delta>0.
\end{equation}
These ensure it is energetically unfavorable for the flow to split into multiple solutions at infinity due to tail interactions (see Lemmas \ref{lemma difMbubble} and \ref{lemma weakinter}). Hypotheses \eqref{eq KPP}--\eqref{eq concat0} are satisfied by \eqref{eq tildef} and by $f(t) = t - t^p$, $p \in (1, (n+2)/(n-2))$, considered in \cite{cortazar1999uniqueness}.  
   A generalized version of Theorem \ref{thm1} for reaction terms satisfying these general hypotheses is given in Theorem \ref{thm gen} at the end of Section \ref{sec background}.
\end{remark}

\begin{remark}\label{reramrk convergence}(Convergence to a ground state as a threshold phenomena)   	 
	 Global solutions of \eqref{eq generalflow} converging to a ground state are rare; they arise as threshold trajectories separating solutions that vanish at infinity from those that blow up in finite time. More precisely, \cite[Theorem 1.1]{feireisl1997threshold} shows that for any compactly supported $u_0 \in W^{1,2}(\R^n)$ there exists $\alpha(u_0) \in (0,\infty)$ such that the solution with initial datum $\alpha u_0$ satisfies:
\begin{itemize}
\item vanishes as $t\to \infty$ if $\alpha \in [0,\alpha(u_0))$;
\item converges to a ground state and satisfies \eqref{eq nonzero} if $\alpha = \alpha(u_0)$;
\item blows up in finite time if $\alpha > \alpha(u_0)$.
\end{itemize}
Our results extend this trichotomy to initial data without compact support; the proof follows exactly as in \cite{feireisl1997threshold}.
\end{remark}

\subsection{Extension to volume constrained flows.} The strategy used in the proof of Theorem \ref{thm1} can be adapted to establish convergence for gradient flows associated with volume-constrained variational problems in $\R^n$, even when the resulting equation is non-autonomous due to the presence of Lagrange multipliers. 

As a concrete example, consider the volume-constrained problem
\begin{equation}
\label{Psi eps m}
\Psi(\varepsilon,m) = \inf \Bigg\{\mathcal{AC}_\e(u): \int_{\R^n} V(u) = m, \, u \in L^1_{\rm loc}(\R^n;[0,1]) \Bigg\}, \qquad \varepsilon, m>0,
\end{equation}
where the energy functional is the Allen--Cahn energy
\begin{equation}
\label{AC eps}
\mathcal{AC}_\e(u) = \varepsilon \int_{\R^n} |\nabla u|^2 + \frac{1}{\varepsilon} \int_{\R^n} W(u), \qquad \varepsilon>0.
\end{equation}
Here $W$ is a non-degenerate double-well potential, $W:[0,1] \to [0,\infty)$, satisfying
\begin{equation}\label{eq usualW}
W(0)=W(1)=0, \quad W>0 \text{ on } (0,1), \quad W''(0), W''(1) >0,
\end{equation}
and the volume potential is given by $V(t) = \big(\int_0^t \sqrt{W}\,ds \big)^{n/(n-1)}$. Problem \eqref{Psi eps m} provides a diffuse interface approximation to the Euclidean isoperimetric problem (when $\varepsilon \ll m^{1/n}$) and has been studied in detail in \cite{maggi2024uniform}. The particular choice of $V$ is linked to the classical isoperimetric inequality; its key properties are discussed extensively in the introductions of \cite{maggi2024uniform, bonforte2024asymptotic}.  

The gradient flow associated with \eqref{Psi eps m} is given by the Cauchy problem
\begin{equation}
\label{diffused VPMCF} 
\left\{
\begin{aligned}
&\partial_t u = 2 \Delta u - \frac{1}{\varepsilon^2} W'(u) + \frac{\lambda_\varepsilon[u(t)]}{\varepsilon} V'(u), \quad \text{in } (0,\infty)\times \R^n , \\
&u(0) = u_0 \in W^{1,2}(\R^n;[0,1]),
\end{aligned}
\right.
\end{equation}
where the Lagrange multiplier $\lambda_\varepsilon[u]$ is defined by
\begin{equation*}
\lambda_\varepsilon[v] := \frac{\int_{\R^n} 2 \varepsilon^2 |\nabla v|^2 V''(v) + W'(v) V'(v)}{\varepsilon \int_{\R^n} (V'(v))^2}.
\end{equation*}
This choice ensures preservation of the volume constraint:
\begin{equation*}
\int_{\R^n} V(u(t)) = \int_{\R^n} V(u_0), \quad \forall t>0.
\end{equation*}
The flow \eqref{diffused VPMCF} is the natural gradient flow associated with \eqref{Psi eps m}, in the sense that
\begin{equation*}
\frac{d}{dt} \mathcal{AC}_\e(u(\cdot,t)) = - \varepsilon \int_{\R^n} (\partial_t u)^2.
\end{equation*}

From the perspective of classical parabolic theory, \eqref{diffused VPMCF} exhibits several peculiar features. It is a non-autonomous semilinear PDE, where the non-autonomy stems from the Lagrange multiplier $\lambda_\varepsilon[u(t)]$. Nevertheless, the underlying variational structure provides sufficient rigidity to establish a subsequential bubbling result analogous to \eqref{eq bubbling}. Moreover, when $u_0$ is compactly supported, one can prove full convergence of solutions to minimizers of \eqref{Psi eps m} for $m = \int_{\R^n} V(u_0)$ and $\varepsilon$ satisfying $\varepsilon << m^\frac{1}{n}$; see \cite[Theorems 1.3--1.4]{bonforte2024asymptotic}.  

Observing that, by \eqref{eq usualW}, the nonlinearity
\[
f(t,s) := \frac{1}{\varepsilon^2} W'(s) - \frac{\lambda_\varepsilon[u(t)]}{\varepsilon} V'(s)
\]
satisfies the expansion
\[
f(t,s) = C s + O(s^{1+\beta}), \quad C>0,
\]
we see that it formally falls into the class of nonlinearities treated by Theorem \ref{thm1} and is therefore amenable to the same proof strategy. More precisely, if $W \in C^3[0,1]$ satisfies the standard non-degeneracy conditions together with
\begin{equation}\label{eq assumtion convexity}
W''(0) \ge \frac{W'(t)}{t}, \quad \forall t \in [0,1], 
\qquad \text{and} \qquad W'''(0) < 0,
\end{equation}
which correspond to \eqref{eq KPP} and \eqref{eq concat0}, then for $n \ge 2$ and $\varepsilon$ sufficiently small (depending only on $\mathcal{AC}_\e(u_0)$, $m$, $W$, and $n$), one can prove full convergence of solutions of \eqref{diffused VPMCF} to a unique minimizer of \eqref{Psi eps m}, without restrictions on the support of $u_0$. In this case, the convergence is exponential as in \eqref{eq nonzero}, with constants degenerating as $\varepsilon \to 0$.

We do not pursue the proof of this statement here. Although the overall strategy parallels that of Theorem~\ref{thm1}, the presence of the Lagrange multiplier $\lambda_\e[u(t)]$ creates a structural discrepancy between the parabolic equation satisfied by $u$ and the elliptic equation satisfied by the minimizers of \eqref{AC eps}. Addressing this mismatch requires several minor but technically involved modifications of the argument.

\subsection{Strategy of the proof}\label{ss:mainresults}

Our argument is based on two main ideas: 
\begin{enumerate}
    \item[(1)] Controlling the rate of decay of $\Vert \partial_t u\Vert_{L^2(\R^n)}$ as $t \to \infty$ ensures strong convergence of the flow, thereby preventing bubbling.
    \item[(2)] Viewing the solution $u$ to \eqref{eq generalflow} as a nearly critical point of the energy functional $J$ allows us to perform a quantitative stability analysis.
\end{enumerate}
A crucial ingredient for the implementation of this strategy is the non-degeneracy of solutions to the stationary problem \eqref{eq gs}, including uniqueness up to translations. These properties are guaranteed by our choice of the nonlinearity $f$ in Theorem \ref{thm1} (see Section \ref{sec background}).\\

The first point is a general principle for parabolic flows: \eqref{eq nonzero} essentially follows from
\begin{equation}\label{eq linear estimateog} 
   J(u(t,\cdot)) - M J(\xi) \le C \Vert \partial_t u \Vert^2_{L^2(\R^n)},
\end{equation}
as established in Section \ref{Proof thm1}. The second point is a property of Palais--Smale sequences for elliptic problems. Writing \eqref{eq generalflow} as $\nabla J(u) = \partial_t u$, the general concentration-compactness theory implies that as $\partial_t u \to 0$, the solution $u$ approaches a superposition of stationary solutions in the sense of \eqref{eq bubbling}.  

By performing a quantitative stability analysis of critical points of $J$, in the spirit of the theory of functional inequalities \cite{figalli2020sharp, aryan2023stability, deng2025sharp}, we can strengthen \eqref{eq bubbling} to
\begin{equation}\label{eq linearerror}
    \Big\Vert u(t,\cdot) - \sum_{i=1}^M \alpha_i(t) \tau_{x_i(t)}[\xi] \Big\Vert_{W^{1,2}(\R^n)} \le C \Vert \partial_t u \Vert_{L^2(\R^n)},
\end{equation}
for sufficiently large $t$, where $x_i(t) \in \R^n$ and $\alpha_i(t) \in \R$ are suitable parameters, and $\tau_{x_i(t)}[\xi]$\footnote{We set $\tau_{x_0}[v](x)=v(x-x_0)$ for every $x,x_0\in\R^n$ and $v:\R^n\to\R^m$.} are translations of the unique positive solution $\xi$ of \eqref{eq gs}. The linear estimate \eqref{eq linearerror}, combined with precise quantitative bounds on the errors introduced by these parameters, leads directly to \eqref{eq linear estimateog} (see Section \ref{subsec linear stability}).\\

There are several significant differences between our approach and previous stability analyses of critical points in functional inequalities. First, since $f'(0) > 0$ and $f$ is non-homogeneous, estimating the error terms in \eqref{eq linearerror}, namely
\begin{equation}\label{eq error par}
    \sum_{i=1}^M |1-\alpha_i| + \sum_{1 \le i < j \le M} \int_{\R^n} \tau_{x_i(t)}[\xi] \tau_{x_j(t)}[\xi],
\end{equation} 
is more involved than in \cite{figalli2020sharp, aryan2023stability, deng2025sharp}, and requires new ideas, as detailed in Lemma \ref{lemma weakinter}. Second, while quantitative stability for functional inequalities typically yields convergence rates of known unique limiting profiles \cite{ciraolo2018quantitative, figalli2020sharp, de2023stability, bonforte2025stability}, in our setting both the uniqueness of the limit and the convergence rate are established simultaneously.  

Finally, we remark that our approach can also be adapted to prove general linear stability estimates for critical points of $J$, i.e., for non-negative solutions $v$ of $\nabla J(v) = h$ with $h \in H^{-1}(\R^n)$, in the spirit of \cite{figalli2020sharp, aryan2023stability, deng2025sharp}.

\subsection{Organization of the paper}

In Section \ref{sec background} we collect some properties of solutions to \eqref{eq bubbling} required for our analysis. Section \ref{sec proofmain} contains the main body of the paper including the proof of Theorem \ref{thm1}. Finally, in Section \ref{sec interest} we prove some technical results that are used throughout the proofs.

\section{Background material}\label{sec background}

In this section, we recall some qualitative properties of ground state solutions to \eqref{eq gs} relevant for our methods. As mentioned in the introduction, the general class of reaction terms considered in this paper are of the form $f(t) = m^2t+O(t^p)$ as $t\to 0^+$ with $m>0$ and $p>1$. We begin this section by pointing out that the classical results of \cite{gidas1981symmetry} apply in this case, yielding radial symmetry and precise decay rates for solutions to \eqref{eq gs}.

\begin{proposition}\label{prop radialuniqueness} 
    Let $\beta\in (0,1)$, $n\geq 2$ and let $f \in C_{loc}^{1,\beta}([0,\infty))$ satisfying $f(0)=0$, $f'(0)=m^2>0$. Assume that $v$ is a $C^2$ solution to
\begin{equation}\label{eq gsap0}
\begin{cases}
\Delta v = f(v) \quad  \R^n, \\
 v(x) \to 0, \quad  \text{as $|x| \to \infty$}.
\end{cases}
\end{equation} 
Then, $v=\tau_{x_0}[\xi]$ for some $x_0\in\R^n$ with $\xi$ strictly radially decreasing. Moreover, 
	\begin{equation}\label{eq tail zeta}
			\frac{1}{C} \frac{e^{-|x|/m}}{\max\{|x|^\frac{n-1}{2},1\}}	\leq  \xi(x) \leq C\frac{e^{-|x|/m}}{\max\{|x|^\frac{n-1}{2},1\}}
	\end{equation}
	and we have that given any $r_0>0$ there exists $C(r_0)$ such that for $r> r_0$
		\begin{equation}\label{eq tail zeta'} 
		\frac{1}{C(r_0)} \xi(x)\leq  -\xi'(|x|) \leq C(r_0)\xi(x) \quad \mbox{  for $|x|\geq r$}.
	\end{equation}
\end{proposition}
\begin{proof}
   By \cite[Theorem 2]{gidas1981symmetry} we have that $v=\tau_{x_0}[\xi]$ for some $\xi$ strictly radially decreasing with respect to some point $x_0$. Furthermore, the same result guarantees
	\begin{equation}\label{eq exp bound GNN}
		\frac{1}{C} \frac{e^{-|x|/m}}{\max\{|x|^\frac{n-1}{2},1\}}	\leq  \xi(x) \leq C \frac{e^{-|x|/m}}{\max\{|x|^\frac{n-1}{2},1\}}.
	\end{equation}
     Consider $w(r) = r^\frac{n-1}{2}\xi(r)$ -this is the so-called Emden-Fowler change of variables. A direct computation in spherical coordinates (see, e.g., \cite[(3-36)]{maggi2024uniform}) shows that
    \begin{equation}\label{eq EF}
w''(r) = w\Big(\frac{a(a-1)}{r^2}+\frac{f(w)}{w}\Big),
    \end{equation}
where $a= \frac{n-1}{2}$. Since  $w(r) \sim e^{-r/m}$ in virtue of \eqref{eq exp bound GNN} and $f(w)= f'(0)w+O(w^\beta)$, we have that for $r_0$ large enough (depending only on $f$ and $n$, we have that
	\begin{equation}\label{eq comparison w}
 \frac{1}{C}	w(r) \leq	w''(r) \leq C w(r) \quad \mbox{if $r\geq r_0$}.
	\end{equation}
	
	So, by integration we also have that $-w'(r)$ is comparable with $e^{-\frac{r}{m}}$ and since
	\begin{equation*}
		-\xi'(r) = \frac{-w'(r)}{r^\frac{n-1}{2}} +\frac{n-1}{2}\frac{w(r)}{r^\frac{n+1}{2}},
	\end{equation*}
	we readily deduce \eqref{eq tail zeta'}.
\end{proof}

The existence of ground-state solutions to \eqref{eq gs} is guaranteed by the results in \cite{berestycki1980existence} under nearly optimal hypotheses for $n\geq 2$. Of particular relevance for our work is the fact that these ground states can be obtained as minimizers of a constrained problem in $\R^n$, at least when $n\geq 3$. We state a slightly weaker version of their theorem, see \cite[Theorem 1.1 and Section 1.3]{berestycki1980existence}.
\begin{proposition}\label{thm BL}
   Let $n\geq 3$. Let us assume that $f \in C_{loc}^{1,\beta}(\R)$ with $f(0)=0$, $f'(0)>0$, $F(s_0)<0$ for some $s_0>0$, and that $\liminf_{t\to \infty} \frac{f(t)}{t^\frac{n+2}{n-2}}\geq 0$. Then \eqref{eq gsap0} has a positive solution $\xi$ that is radially decreasing and, after rescaling, solves the constrained problem
   \begin{equation}\label{eqBL}
       \inf\left\{ \int_{\R^n}|\nabla v|^2 \, \Big| \int_{\R^n}F(v) =-\delta \right\},
   \end{equation}
   for some $\delta>0$. Furthermore, $J(\xi)>0$.
\end{proposition}

By combining Proposition \ref{prop radialuniqueness} and Proposition \eqref{thm BL}, uniqueness of solutions of \eqref{eq gsap0} reduces to an ODE problem. There are several uniqueness resutls that hold under fairly general hypotheses on $f$- see, e.g., \cite{chen1991uniqueness, mcleod1993uniqueness}. In the interest of our example, we remark that \eqref{eq gsap0} has a unique solution if $f$ has the form \eqref{eq tildef}, see \cite[Example 2]{chen1991uniqueness}.\\

We also recall some properties of the spectrum of the second variation of $D^2 J(\xi)$.

\begin{proposition}\label{thm nondegneracy}
    Let $n\geq 3$ and assume that $f$ has the form \eqref{eq tildef}. Let $\xi$ be the unique radially symmetric solution of \eqref{eq gsap0} and consider the quadratic form
    \begin{equation*}
         Q(\varphi, \varphi):=  D^2J(\xi)(\varphi, \varphi)=\int_{\R^n} |\nabla \varphi|^2+f'(\xi)\varphi^2.
    \end{equation*}
    $Q$ satisfies the following properties:
    \begin{enumerate}
        \item Given any $\psi \in W^{1,2}(\R^n)$
          \begin{equation}\label{eq svxi'}
        Q(\xi', \psi)= - (n-1) \int_{\R^n} \frac{\xi' \psi}{|x|^2}dx,
    \end{equation}
    in particular $Q(\xi', \xi') <0$ and $Q(\xi', \xi)>0$.
        \item $Q$ has only one negative eigenvalue and its kernel is spanned by $\{\partial_i \xi\}_{i=1}^n$. 
        \item There exists $C>0$ such that 
        \begin{equation}\label{eq nondegJ}
      Q(\varphi, \varphi) \geq \frac{1}{C} \int_{\R^n} |\nabla \varphi|^2 +\varphi^2
    \end{equation}
    for $\varphi \in W^{1,2}(\R^n)$ satisfying
    \begin{equation}\label{eq orthoJ}
      Q(\varphi, \xi') =0, \quad \int_{\R^n} \varphi \partial_j \xi =0 \quad \forall \, j=1, \cdots , n.
    \end{equation}
    \end{enumerate}    
\end{proposition}
\begin{proof}
We rewrite first $\Delta \xi=f(\xi)$ in spherical coordinates so that $\xi''(r) + \frac{(n-1)}{r}\xi' -f(\xi)(r)=0$ for $r>0$. Differentiating the latter equation in $r$ yields
    \begin{equation}\label{eq xi'}
        \Delta \xi'-f(\xi)\xi'= \xi''' + \frac{(n-1)}{r}\xi'' -f'(\xi)\xi'= \frac{(n-1)}{r^2}\xi',
    \end{equation}
    multiplying on both sides by $\psi \in W^{1,2}(\R^n)$ and integrating by parts yields \eqref{eq svxi'}. Notice that since $n\geq 3$, $\psi/|x| \in L^2(\R^n)$ in virtue of Hardy's inequality. So, by plugging $\xi$ and $\xi'$ into \eqref{eq svxi'} and by recalling that $\xi'\leq 0$ and $\xi>0$, we deduce that $Q(\xi', \xi') <0$ and $Q(\xi', \xi)>0$. This shows part (1).\\

   Regarding part (2), in virtue of \cite[Proposition 4.1]{feireisl1997threshold} have that the kernel of $Q$ is generated by $\partial_i \xi$ -see also \cite{ni1993locating}. On the other hand, in light of Proposition \ref{thm BL} $\xi(x)= \zeta(\lambda^\frac{-1}{2} x)$ where $\lambda>0$ and $\zeta$ solving the problem \eqref{eqBL}. Furthermore, from the minimality of $\zeta$ we have that
   \begin{equation}\label{eq zeta}
       \Delta \zeta = \lambda f(\zeta),
   \end{equation}
   and
   \begin{equation}\label{eq coerQ}
       \tilde{Q}(\psi, \psi) = \int_{\R^n} |\nabla \psi|^2 + \lambda f'(\zeta)\psi^2 dx \geq 0,
   \end{equation}
   for $\psi \in W^{1,2}(\R^n)$ satisfying $\int_{\R^n} f(\zeta) \psi =0$. Set $\varphi(x) = \psi(\sqrt{\lambda} x)$, by a simple change of variables \eqref{eq coerQ} translates into 
      \begin{equation}\label{eq coerQ}
       Q(\varphi, \varphi) \geq 0\qquad \mbox{if } \quad \int_{\R^n} f(\xi)\varphi=0.
   \end{equation}
   This observation combined with $Q(\xi', \xi')<0$ implies that $Q$ has a one dimensional negative eigen-space. \\

   Lastly, part (3) is a direct consequence of part (2) combined with standard spectral theory.
\end{proof}

\begin{remark}\label{remark dimres}
    The only place where the dimensional constraint plays a role ($n\geq 3$) is in the variational characterization of $\xi$, i.e., \eqref{eqBL}. This characterization does not hold in dimension two since in this dimension $\int_{\R^n} F(\xi)=0$ because of the Pohozaev inequality, see \cite{berestycki1980existence}. There are other cases where the ground state solution can be realized as a minimizer of a volume constrained problem in any dimension, e.g. \cite{bonforte2024asymptotic}, in which case the same argument presented so far follows mutatis mutandis.
\end{remark}

We now restate (an strengthened version of) \eqref{eq bubbling} under the hypotheses of Theorem \ref{thm1}. As mentioned in the introduction, results of this type are rather general and the one presented below is a corollary from the results in \cite{lions1988positive}, see also \cite[Lemma 3.3]{feireisl1997threshold}.

\begin{proposition}\label{eq Jbubbling}
Let $n\geq 3$, and let $f$ and $u$ as in Theorem \ref{thm1}. Furthermore, assume that
\begin{equation}\label{eq limbubbles}
   M =  \lim_{t \to \infty} \frac{J(u(t))}{J(\xi)} >0.
\end{equation}
Then, given any sequence $\{T_k\}_{k\in \N}$ with $\lim_{k \to \infty} T_k =\infty$, there exist a subsequence $\{t_k\}_{k\in \N}$ and centers $x_k^i \in \R^n$ for $k\in \N$ and $i \in \{1,\cdots, M\}$ such that 
    \begin{eqnarray}\notag
   &&\lim_{k\to \infty} \Big\Vert u_k -\tau_{x^0}\xi -\sum_{i=1}^M \tau_{x_k^i}[\xi] \Big \Vert_{W^{2,2}\cap C^2(\R^n)} =0,\\ \label{eq bubbling1}
   &&\lim_{k\to \infty} |x^i_k -x^j_k| \to \infty \qquad \mbox{for $i\neq j$, $i,j =1, \cdots, M,$}
\end{eqnarray}
where $u_k(x):= u(t_k, x)$.
\end{proposition}

We close this section stating a generalized version of Theorem \ref{thm1}.

\begin{theorem}\label{thm gen}
    Let $\beta \in (0,1)$, $n\geq 3$, $u_0 \in W^{1,2}(\R^n)$ with $u_0 \geq 0$ almost everywhere in $\R^n$, and let $u$ be a solution to \eqref{eq generalflow}. Let $f \in C^{1,\beta}([0, \infty)$ with $f(0)=0$, $f'(0)>0$, such that \eqref{eq gsap0} has a unique solution (up to translations), and that the conclusions of Proposition \ref{thm nondegneracy} and Proposition \ref{eq Jbubbling} hold. If, furthermore, $f$ satisfies \eqref{eq KPP}  and \eqref{eq concat0}, then, $u$ is the unique solution to \eqref{eq generalflow}, $u(t,x)>0$ for $t>0$, and 
\begin{equation}\label{eq nonzerogen}
   J(u(t, \cdot))-J(\xi) \leq C e^{-\frac{t}{C}}, \qquad \Vert u(t, \cdot)-\tau_{x^*}[\xi]\Vert_{C^2(\R^n)\cap W^{2,2}(\R^n)} \leq C e^{-\frac{t}{C}},
\end{equation}
where $\xi$ is the unique ground state solution of $\Delta \xi=f(\xi)$ and $x^* \in \R^n$.
\end{theorem}

\section{Proof of the main results}\label{sec proofmain}

\subsection{Setup.}\label{subsec setup} Let $f$ and $u$ as in Theorem \ref{thm gen} , and let $M\in \N$ as in \eqref{eq limbubbles}. We define the set of $M$-bubbles as 
$$\mathbb{B}_{M} = \Big\{\sum_{i=1}^M  \alpha_i\tau_{x^i}\big[\xi\big]\, \Big|\,  x^i \in \R^n, \alpha_i \in \R\Big\},$$
where $\xi$ is the unique radial ground-state solution of $\Delta \xi = f(\xi)$, $\alpha_i \in \R$, and $x^i \in \R^n$. We also consider the space of simple or unweighted $M$-bubbles 
$$\mathbb{SB}_{M} = \Big\{\sum_{i=1}^M  \tau_{x^i}\big[\xi\big]\, \Big|\,  x^i \in \R^n\Big\}.$$
Given $\nu >0$, we say that a $M$-bubble $\eta(t)=\sum_{i=1}^M  \alpha_i \tau_{x^i}\big[\xi\big]$ is \textit{$\nu$-interacting} if
	\begin{equation}\label{eq nuweak interation}
		\sum_{1\leq i< j \leq M}\int_{\R^n} \tau_{x^i}\big[\xi\big] \tau_{x^j}\big[\xi\big] \leq \nu.
	\end{equation}

Given $t>0$, we consider the minimization problem
	\begin{equation}\label{eq bestmatching}
		\Gamma( t)=	\inf \Big\{  \Vert u(t,\cdot)- \theta \Vert_{L^2(\R^n)} \, |\, \theta \in \mathbb{SB}_M \Big\}.
	\end{equation}
The existence of a minimizer $\theta(t) = \sum_{i=1}^M \tau_{x^i(t)}\big[\xi\big]$ is guaranteed by Lemma \ref{lemma exbm} provided that $t\geq T$ for $T$ sufficiently large. Associated to $\theta$, we define a best matching (weighted) $M$-bubble as
\begin{equation*}
    \eta(t) = \sum_{i=1}^M \alpha_i(t) \tau_{x^i(t)}\big[\xi\big],
\end{equation*}
where the $\alpha_i$ are given by solutions to the linear system
\begin{equation}\label{eq definitionalpha}    D^2J(\tau_{x^i(t)}\big[\xi'\big])\Bigg(\tau_{x^i(t)}\big[\xi\big], u-\sum_{j=1}^M \alpha_j(t) \tau_{x^j(t)}\big[\xi\big]\Bigg)=0, \quad \mbox{for $i=1,\cdots, M$}.
\end{equation}
The existence and uniqueness of $\alpha_1(t), \cdots, \alpha_M(t)$ for $t\geq T$ is also discussed in Lemma \ref{lemma exbm}. In virtue of Proposition \ref{eq Jbubbling}, we have $\lim_{t\to \infty} \Gamma(t)=0$. Associated with $\eta(t)$  we have centers $(x^1(t),\cdots, x^M(t))$ that drift apart as $t\to \infty$, which, by the properties of $\xi$ (see Lemma \ref{lemma pairwise inter}) imply that $\eta(t)$ is $\nu(t)$-interacting with $\lim_{t \to \infty} \nu(t)=0$. In particular, we have that $\tau_{x^j(t)}\big[\xi\big]$ and $\tau_{x^i(t)}\big[\xi'\big]$ have supports asymptotically disjoint as $t\to \infty$ (if $j\neq i$) which, heuristically, implies that \eqref{eq definitionalpha}  yields
\begin{equation}\label{eq limitalpha}
    \lim_{t\to \infty} \alpha_i(t) = \lim_{t\to \infty} \frac{D^2J(\tau_{x^i(t)}\big[\xi\big])(\tau_{x^i(t)}\big[\xi'\big],u)}{D^2J(\tau_{x^i(t)}\big[\xi\big])(\tau_{x^i(t)}\big[\xi'\big],\tau_{x^i(t)}\big[\xi\big])}=1,
\end{equation}
where we critically use that $D^2J(\xi)(\xi', \xi)>0$, see Proposition \ref{thm nondegneracy}. The convergence \eqref{eq limitalpha} is rigorously justified in  Lemma \ref{lemma exbm}. It follows from the previous considerations that  given $\delta>0$, we can always take $T=T(\delta)>0$ such that
\begin{equation}\label{eq smallGammanu}
    \Gamma(t) + \nu(t)+\sum_{i=1}^M |1-\alpha_i(t)| \leq \delta, \quad t\geq T.
\end{equation}

Given a best-matching $M$-bubble $\eta(t)$, we define $\rho(t) = u(t, \cdot)- \eta(t)$. Notice that since $\rho(t) \to 0$ in $L^2(\R^n)$, parabolic regularity implies that $\rho(t) \to 0$ in $C^1(\R^n)$ so that we can assume that $T$ is large enough for $\Vert \rho(t)\Vert_{C^1(\R^n)} \leq \delta$. We will usually drop the dependence on $t$ in $\rho(t), \eta(t), \theta(t)$ and in the centers $x^i(t)$ and $\alpha^i(t)$ when is understood that they correspond to a best-matching $M$-bubble.\\

The minimality of $\theta$ in \eqref{eq bestmatching} implies the following set of orthogonality conditions
	\begin{equation}\label{eqn orthogonality 0}
		\int_{\R^n} (u-\theta) \partial_j \tau_{x^i(t)}[\xi]=0,
	\end{equation}
 for $i=1,\cdots M$ and $j=1\cdots, n$. In turn, this implies the set of almost-orthogonality conditions for $\rho$
 \begin{equation}\label{eqn orthogonality 0}
		\Bigg|\int_{\R^n} \rho \partial_j \tau_{x^i}[\xi]\Bigg|\leq C\sum_{l=1}^M|1-\alpha_i| +C\sum_{1\leq l< m \leq M}\int_{\R^n} \tau_{x^l}\big[\xi\big] \tau_{x^m}\big[\xi\big].
	\end{equation}
 
 On the other hand, \eqref{eq definitionalpha} can be rewritten in terms of $\rho$ as
 \begin{equation}\label{eq orthpxi} 
 D^2J(\tau_{x^i(t)}\big[\xi\big]) (\tau_{x^i}\big[\xi'\big],\rho)=0, \quad \mbox{for $i=1,\cdots, M$}.
\end{equation}
This last orthogonality condition is required to guarantee the coercivity of the second variation of $J$ in the direction $\rho(t)$ (i.e., Proposition \ref{thm nondegneracy}). It will also play a central role in bounding $\sum_{l=1}^M|1-\alpha_i|$ so that \eqref{eqn orthogonality 0} becomes a suitably quantitative almost-orthogonality condition --see Lemma \ref{cor stabilitydich}.

\subsubsection{The $M$-bubble energy inequality and the role of assumption \eqref{eq concat0}}

Our first step is to derive an estimate for the energy deficit $J(u(\cdot, t)) -M J(\xi)$, which will allows us to link $\partial_t u$ to the distance between $u$ and a best-matching $M$-bubble $\eta(t)=\sum_{i=1}^M \alpha_i(t)\tau_{x^i(t)}\big[\xi\big]$. We point out that the derivation of this estimate (i.e., Lemma \ref{lemma difMbubble}) is the only place where hypothesis \eqref{eq concat0} is used.\\

As a preparatory ingredient, we recall that since $t\to J(u( t, \cdot))$ is non-increasing, we have
 \begin{equation}\label{eq bee}
     J(u(\cdot, t)) \geq M J(\xi) = \sum_{i=1}^M J\big(\tau_{x^i(t)}\big[\xi\big]\big),
 \end{equation}   
 where the last equality follows from the translation invariance of $J$.

\begin{lemma}\label{lemma difMbubble}
    Let $u$ be as in Theorem \ref{thm1} and assume \eqref{eq limbubbles}. Let $\eta(t) = \sum_{i=1}^M \alpha_i(t) \tau_{x^i(t)}\big[\xi\big]$ be a best-matching $M$-bubble. There exists $T>0$ large such that for $t\geq T$
    \begin{eqnarray}\notag
J(u)- MJ(\xi)+\int_{\R^n}  2|\nabla \rho(t)|^2 + f'(\eta(t)) \rho(t)^2&\leq& C\Vert \partial_t u\Vert_{L^2(\R^n)} \Vert \rho(t)\Vert_{L^2(\R^n)}\\\notag
&&+C\sum_{i=1}^M |1-\alpha_i|^2\\\notag
&&+C\sum_{1\leq i <j\leq M} \Big(\int_{\R^n}\tau_{x^i(t)}\big[\xi\big]\tau_{x^j(t)}\big[\xi\big]\Big)^2\\\label{eq keyestimate}
	&&+\int_{\R^n}|\rho(t)|^{2+\beta},
\end{eqnarray}
\end{lemma}
\begin{proof}

\medskip

\noindent{\it Preparations:}  Throughout the proof, we assume $T$ large enough such that \eqref{eq smallGammanu} holds and $|x^i(t)-x^j(t)| \geq \frac{1}{\delta}$ for $i\neq j$, where $\delta>0$ is suitably small. Define
\begin{equation}\label{eq defr}
    r(t)=\frac{1}{2}\min\{|x^i(t)-x^j(t)| \, |\, 1\leq i<j\leq M\}.
\end{equation} 
Then the balls $\{B_{r(t)}(x^i(t))\}_{i=1}^M$ are pairwise disjoint and drift apart as $t \to \infty$. For notational convenience, we suppress the explicit $t$-dependence of $r(t)$, the best-matching $M$-bubble $\eta(t)$ and its parameters for the rest of the proof.

\medskip

\noindent{\it Step 1:} We derive a collection of auxiliary integral inequalities in the balls $B_r(x^i)$ and in their complements.

\medskip

Let $i,j \in \{1,\dots,M\}$ with $i \neq j$.
We begin by multiplying
\begin{equation}\label{eq onebubbleeq}
\Delta \tau_{x^i}[\xi] = f\big(\tau_{x^i}[\xi]\big)
\end{equation}
by $\tau_{x^j}[\xi]$ and integrating by parts over $B_r(x^i)$ to obtain
\begin{eqnarray}\notag
\int_{B_r(x^i)} \nabla \tau_{x^i}[\xi]\cdot \nabla \tau_{x^j}[\xi]
&=& -\int_{B_r(x^i)} f(\tau_{x^i}[\xi]) \tau_{x^j}[\xi] \label{eq weak intergrad balls}\\
&&+ \xi'(r)\int_{\partial B_r(x^i)} \tau_{x^j}[\xi],
\end{eqnarray}
where we have used the decay of $\xi$ together with its radial symmetry. Set $B_{i,j} := B_r(x^i) \cup B_r(x^j)$. Proceeding as above, we multiply \eqref{eq onebubbleeq} by $\tau_{x^j}[\xi]$ and integrate by parts over $B_{i,j}^c$ to deduce
\begin{eqnarray}\label{eq weak intergrad balls2}
\int_{B_{i,j}^c} \nabla \tau_{x^i}[\xi]\cdot \nabla \tau_{x^j}[\xi]
&=& - \int_{B_{i,j}^c} f(\tau_{x^i}[\xi]) \tau_{x^j}[\xi] \notag\\
&&- \xi'(r)\int_{\partial B_r(x^i)} \tau_{x^j}[\xi]
- \xi(r) \int_{\partial B_r(x^j)} \partial_{\nu}\tau_{x^i}[\xi].
\end{eqnarray}
Since $f(0)=0$ and $f'(0)>0$, there exists $\delta_0>0$ such that $f>0$ on $(0,\delta_0)$. In particular, taking $T$ sufficiently large ensures that $f(\tau_{x^i}[\xi])>0$ in $B_r(x^j)$. Applying the divergence theorem to \eqref{eq onebubbleeq}, we obtain
\begin{equation}\label{eq sign f}
\int_{\partial B_r(x^j)} \partial_{\nu}\tau_{x^i}[\xi]
= \int_{B_r(x^j)} \Delta \tau_{x^i}[\xi]
= \int_{B_r(x^j)} f(\tau_{x^i}[\xi]) > 0.
\end{equation}

Adding \eqref{eq weak intergrad balls} for $(i,j)$ and $(j,i)$ together with \eqref{eq weak intergrad balls2} and combining it with \eqref{eq sign f}, we deduce for $i<j$ that
\begin{eqnarray}\notag
\int_{\mathbb{R}^n} \nabla \tau_{x^i}[\xi]\cdot \nabla \tau_{x^j}[\xi]
&<& -\int_{B_r(x^j)} f(\tau_{x^j}[\xi]) \tau_{x^i}[\xi] \notag\\
&&- \int_{B_r(x^i)} f(\tau_{x^i}[\xi]) \tau_{x^j}[\xi] \notag\\
&&- \int_{B_{i,j}^c} f(\tau_{x^i}[\xi]) \tau_{x^j}[\xi],
\label{eq inter gradient}
\end{eqnarray}
where we have exploited that $\xi'<0$ to control the term $\xi'(r)\int_{\partial B_r(x^j)} \tau_{x^i}[\xi]$.

By Proposition \ref{prop radialuniqueness}, we may choose $R>1$ sufficiently large so that $\xi \le \delta_0$ in $B_R^c$. By \eqref{eq defr}, taking $T$ sufficiently large ensures $r \ge R$, and hence $f(\xi)>0$ and $F(\xi)>0$ on $B_r^c$. Define
\[
C_{r} := \bigcap_{i=1}^M B_r(x^i)^c.
\]
Since $C_r \subset B_{i,j}^c$ for every pair $(i,j)$, summing \eqref{eq inter gradient} over $i<j$ yields
\begin{align}\notag
\sum_{1\le i<j\le M} \int_{\mathbb{R}^n}
\nabla \tau_{x^i}[\xi]\cdot \nabla \tau_{x^j}[\xi]
&\le - \sum_{i=1}^M \int_{B_r(x^i)} f(\tau_{x^i}[\xi])
\Big(\sum_{j\neq i} \tau_{x^j}[\xi]\Big) \notag\\
&\quad - \sum_{1\le i<j\le M}
\int_{C_r} f(\tau_{x^i}[\xi]) \tau_{x^j}[\xi].
\label{eq added grad inter}
\end{align}

Next, we estimate $F(\eta)-\sum_{i=1}^M F(\tau_{x^i}[\xi])$ by applying a second-order Taylor expansion of $F$ around $\tau_{x^i}[\xi]$ in $B_r(x^i)$ for $i=1,\dots,M$, which gives
\begin{align}\notag
\int_{\mathbb{R}^n} F(\eta)
- \sum_{i=1}^M F(\tau_{x^i}[\xi])
&= \sum_{i=1}^M \int_{B_r(x^i)}
f(\tau_{x^i}[\xi])(\eta-\tau_{x^i}[\xi]) \notag\\
&\quad +  \int_{B_r(x^i)}\sum_{i=1}^M
O\big((\eta-\tau_{x^i}[\xi])^2\big)- \sum_{j \neq i} F(\tau_{x^j(t)}\big[\xi\big]) \notag\\
&\quad + \int_{C_r}
\Big(F(\eta)-\sum_{i=1}^M F(\tau_{x^i}[\xi])\Big).
\label{eq added potential}
\end{align}

Since $F(\tau_{x^j}[\xi])>0$ in $B_r(x^i)$ for $j\neq i$, we deduce
\begin{align}\notag
\int_{\mathbb{R}^n} F(\eta)
- \sum_{i=1}^M F(\tau_{x^i}[\xi])
&\le \sum_{i=1}^M \int_{B_r(x^i)}
f(\tau_{x^i}[\xi])
\Big[(\alpha_i-1)\tau_{x^i}[\xi]
+ \sum_{j\neq i}\alpha_j \tau_{x^j}[\xi]\Big] \notag\\
&\quad + C\sum_{i=1}^M
\int_{B_r(x^i)} (\eta-\tau_{x^i}[\xi])^2 \notag\\
&\quad + \int_{C_r}
\Big(F(\eta)-\sum_{i=1}^M F(\tau_{x^i}[\xi])\Big).
\label{eq added potential2}
\end{align}

As a last preparatory observation, we notice that multiplying \eqref{eq onebubbleeq} by $\tau_{x^i}[\xi]$ and integrating by parts over $\mathbb{R}^n$ yields.
\begin{equation}\label{eq diff}
\int_{\mathbb{R}^n} \big|\nabla \tau_{x^i}[\xi]\big|^2
= - \int_{\mathbb{R}^n}
f(\tau_{x^i}[\xi]) \tau_{x^i}[\xi].
\end{equation}

\medskip

\noindent{ \it Step 2:} We estimate $J(\eta)-\sum_{i=1}^M J(\tau_{x^i}\big[\xi\big])$.\\

Combining \eqref{eq added grad inter}, \eqref{eq added potential2}, and \eqref{eq diff}, rearranging terms, and exploiting the uniform boundedness of the $\alpha_i$ for $t\ge T$ ensured by \eqref{eq smallGammanu}, we obtain \begin{eqnarray}\notag 
J(\eta)-\sum_{i=1}^M J(\tau_{x^i}\big[\xi\big])&=& \sum_{i=1}^M\int_{\R^n}(\alpha_i^2-1)\Big|\nabla\tau_{x^i}\big[\xi\big]\Big|^2 +\sum_{1\leq i < j \leq M}2 \alpha_i \alpha_j \nabla\tau_{x^i}\big[\xi\big]\cdot \nabla \tau_{x^j}\big[\xi\big]\\\notag
&& +\int_{\R^n}F (\eta)- \sum_{i=1}^M F(\tau_{x^i}\big[\xi\big])\\\notag && \leq \sum_{i=1}^M C(\alpha_i-1)^2\int_{B_r(x^i)} \big|f (\tau_{x^i}\big[\xi\big])\tau_{x^i}\big[\xi\big]\big|+ |\alpha_i-1| \int_{B_r(x^i)^c} \big|f (\tau_{x^i}\big[\xi\big])\tau_{x^i}\big[\xi\big] ]\big| \\\notag &&+ C\sum_{i=1}^M \int_{B_r(x^i)} |1-\alpha_i| \sum_{j\neq i} | \tau_{x^j}\big[\xi\big] f (\tau_{x^i}\big[\xi\big])|\\ \notag && + \int_{C_r} F (\eta)- \sum_{i=1}^M F(\tau_{x^i}\big[\xi\big])- \sum_{1\leq i < j \leq M} \alpha_i \alpha_j f(\tau_{x^i}\big[\xi\big]) \tau_{x^j}\big[\xi\big]\\ \label{eq added grad inter} &&+ C\int_{B_r(x^i)} (\eta-\tau_{x^i}\big[\xi\big])^2. \end{eqnarray}
We first analyze the term in the penultimate line of \eqref{eq added grad inter}. 
By the decay of $\xi$, we may take $T$ sufficiently large (and hence $r$ large) so that $\eta \le \delta$ in $C_{r}$ and Lemma~\ref{lemma tailconcavity} applies. Thus, in $C_r$
\begin{eqnarray}\notag &&F (\eta)- \sum_{i=1}^M F(\tau_{x^i}\big[\xi\big])- \sum_{1\leq i < j \leq M} \alpha_i \alpha_j f(\tau_{x^i}\big[\xi\big]) \tau_{x^j}\big[\xi\big]\\\notag &&= F (\eta)- \sum_{i=1}^M F(\alpha_i\tau_{x^i}\big[\xi\big]) -\sum_{1\leq i < j \leq M} \alpha_j f(\alpha_i\tau_{x^i}\big[\xi\big]) \tau_{x^j}\big[\xi\big]\\\notag && + \sum_{i=1}^M F(\alpha_i\tau_{x^i}\big[\xi\big])-F(\tau_{x^i}\big[\xi\big]) + \sum_{1\leq i < j \leq M} \alpha_j \tau_{x^j}\big[\xi\big](f(\alpha_i\tau_{x^i}\big[\xi\big])-\alpha_i f(\tau_{x^i}\big[\xi\big]))\\\label{eq tailF} && \leq C\sum_{i=1}^M |\alpha_i-1|\tau_{x^i}\big[\xi\big]^2+C\sum_{1\leq i < j \leq M} |\alpha_i-1|\tau_{x^i}\big[\xi\big]\tau_{x^j}\big[\xi\big], \end{eqnarray} 
Here we applied Lemma~\ref{lemma tailconcavity} to conclude that the first term after the equal sign is nonpositive. We also used the uniform boundedness of $|\alpha_i|$ and the estimate $|f(t)|\le Ct$. Combining \eqref{eq tailF} with elementary inequalities, we deduce from \eqref{eq added grad inter} that
\begin{eqnarray}\notag J(\eta)-\sum_{i=1}^M J(\tau_{x^i}\big[\xi\big]) &\leq& C\sum_{i=1}^M (\alpha_i-1)^2+ \Big(\int_{B_r(x^i)^c} \tau_{x^i}\big[\xi\big]^2 \Big)^2\\\label{eq added grad inter2} &&+ C\sum_{1\leq 1 <i<j\leq M} \Big( \int_{B_r(x^i)} \tau_{x^i}\big[\xi\big])\tau_{x^j}\big[\xi\big]\Big)^2.
\end{eqnarray}
We now recall that $r$ is defined in \eqref{eq defr}. Using Proposition~\ref{prop radialuniqueness} together with the interaction estimate \eqref{eq interaction function}, we obtain
\begin{equation*}
    \Big(\int_{B_r(x^i)^c} \tau_{x^i}\big[\xi\big]^2 \Big)^2\leq C\xi(r)^2 \leq C\sum_{1\leq 1 <i<j\leq M} \Big( \int_{B_r(x^i)} \tau_{x^i}\big[\xi\big])\tau_{x^j}\big[\xi\big]\Big)^2,
\end{equation*}
which combined with \eqref{eq added grad inter2} implies 
\begin{eqnarray}\notag
J(\eta)-\sum_{i=1}^M J(\tau_{x^i}\big[\xi\big]) &\leq& C\sum_{i=1}^M (\alpha_i-1)^2\\\label{eq added grad inter22} 
&&+ C\sum_{1\leq 1 <i<j\leq M} \Big( \int_{B_r(x^i)} \tau_{x^i}\big[\xi\big])\tau_{x^j}\big[\xi\big]\Big)^2.
\end{eqnarray}

\medskip

\noindent{ \it Step 3:} We show \eqref{eq keyestimate}.\\

Linearizing $J(\eta)=J(u-\rho)$ around $u$ gives \begin{eqnarray}\notag J(\eta)-J(u)& \geq&\int_{\R^n} |\nabla \rho|^2 + \frac{1}{2} f'(u) \rho^2\\\notag && -\int_{\R^n} 2\nabla u\cdot \nabla \rho+ f(u)\rho \\ \notag && -\int_{\R^n}C|\rho|^{2+\beta}. 
\end{eqnarray} 
Exploiting the $C^{\beta}$ regularity of $f'$, we change $f'(u)$ by $f'(\eta)$ in the previous equation at the expense of enlarging the error $|\rho|^{2+\beta}$. Thus,
\begin{eqnarray}\notag J(\eta)-J(u)& \geq&\int_{\R^n} |\nabla \rho|^2 + \frac{1}{2} f'(\eta) \rho^2\\\notag && -\int_{\R^n} 2\nabla u\cdot \nabla \rho+ f(u)\rho \\ \label{eq linearization u} && -\int_{\R^n}C|\rho|^{2+\beta}.
\end{eqnarray}
Testing \eqref{eq generalflow} with $\rho$ and integrating over $\R^n$ gives
\begin{equation}\label{eq testpar} -\int_{\R^n} 2\nabla u\cdot \nabla \rho - f(u)\rho= \int_{\R^n} \pa_t u \rho. \end{equation} Together \eqref{eq added grad inter22}, \eqref{eq linearization u}, and \eqref{eq testpar} yields \begin{eqnarray*}\notag J(u)- \sum_{i=1}^MJ\Big( \tau_{x^i}\big[\xi\big]\Big)&=& J(\eta)-\sum_{i=1}^M J(\tau_{x^i}\big[\xi\big]) - (J(\eta)-J(u))\\\notag &\leq& -\int_{\R^n} |\nabla \rho|^2 + \frac{1}{2}f'(\eta) \rho^2\\\notag &&+C\sum_{i=1}^M (\alpha_i-1)^2+ C \sum_{1\leq 1 <i<j\leq M} \Big( \int_{\R^n} \tau_{x^i}\big[\xi\big])\tau_{x^j}\big[\xi\big]\Big)^2\\
&&+ C\int_{\R^n}|\rho|^{2+\beta}-\int_{\R^n} \pa_t u \rho,
\end{eqnarray*} 
which proves \eqref{lemma difMbubble}.
\end{proof}

\subsection{Linear stability and energy deficit estimate}\label{subsec linear stability} 
For clarity, we now explain how Lemma~\ref{lemma difMbubble}, together with additional estimates, leads to the proof of the main theorems. More precisely, we show that \eqref{eq linear estimateog} and \eqref{eq linearerror} follow from the following two properties:
\begin{enumerate}
    	\item \textbf{Quantitative weak interaction estimate:} given $\delta>0$, there exists $T(\delta)$ such that if $t\geq T(\delta)$
	
	 \begin{equation}\label{eq quant weak interaction}  
		\sum_{i=1}^M|1-\alpha_i|+\sum_{1\leq i \neq j \leq M}\int_{\R^n}\tau_{x^i(t)}\big[\xi\big]\tau_{x^j(t)}\big[\xi\big]\leq \delta \Vert \rho(t)\Vert_{L^2(\R^n)} + C\Vert \partial_t u\Vert_{L^2(\R^n)}.
	\end{equation}

    \item  \textbf{Multibubble stability dichotomy:} there exists $T>1$ and $\delta_0>0$ such that if $t\geq T$ then either 
    \begin{equation}\label{eq condstab}
        \Vert \partial_t u\Vert_{L^2(\R^n)} \geq \delta_0 \Vert \rho(t)\Vert_{H^1(\R^n)}
    \end{equation} 
    or
	\begin{equation}\label{eq stability Mbubble0}
		\int_{\R^n}   |\nabla \rho(t)|^2 + f'(\eta(t)) \rho(t)^2 \geq \frac{1}{C}\Vert \rho(t)\Vert_{H^1(\R^n)}^2.
	\end{equation}
	
\end{enumerate}
We defer the proofs of these properties to the end of this section (see Lemma~\ref{lemma weakinter} and Corollary~\ref{cor stabilitydich}). We now explain how to combine them with Lemma~\ref{lemma difMbubble} to prove \eqref{eq linear estimateog} and \eqref{eq linearerror}.

\begin{lemma}\label{lemma linearestimate}
    Let $u$ be as in Theorem \ref{thm1} and assume \eqref{eq limbubbles}. Then there exists $T>0$ such that if $t\geq T$    
    \begin{equation}\label{eq linear estimate2}
J(u) -MJ(\xi)\leq C	\Vert \partial_t u\Vert^2_{L^2(\R^n)}. 
\end{equation}
\end{lemma}
\begin{proof}
 Let $T$ be such that Lemma \ref{lemma difMbubble} and the properties (1) and (2) hold with $\delta$ suitably small in \eqref{eq quant weak interaction}. We notice first that this readily implies
 \begin{eqnarray}\label{eq linear estimate1}
	\Vert \rho(t)\Vert_{H^1(\R^n)}\leq C\Vert \pa_t u\Vert_{L^2(\R^n)}.
\end{eqnarray}
Indeed, if \eqref{eq condstab} holds in the dichotomy in property (2), there is nothing to prove. If instead, \eqref{eq stability Mbubble0}
holds, we can combine it with Lemma \ref{lemma difMbubble}, \eqref{eq bee}, and \eqref{eq quant weak interaction} to deduce that \eqref{eq linear estimate1} holds for $t\geq T$.
Next, since
\begin{equation}\label{eq boundsc}
		\int_{\R^n}   |\nabla \rho(t)|^2 + f'(\eta(t)) \rho(t)^2 \leq C\Vert \rho(t)\Vert_{H^1(\R^n)}^2,
	\end{equation}
we can use Lemma \ref{lemma difMbubble} again in combination with \eqref{eq quant weak interaction} to bound this time the energy deficit as follows
\begin{eqnarray*}\notag
	J(u)- MJ(\xi) &\leq& C\Big(\Vert \rho(t)\Vert_{H^1(\R^n)}^2 +\Vert \partial_t u\Vert_{L^2(\R^n)}\Vert \rho(t)\Vert_{H^1(\R^n)} \Big)
\end{eqnarray*}
which combined with \eqref{eq linear estimate1} yields \eqref{eq linear estimate2} as we set out to prove.
\end{proof}
\subsection{Proof of Theorem \ref{thm1}}\label{Proof thm1}

We show how to Lemma \ref{lemma linearestimate} yield rates of convergence for $J(u)$ and how this, in turn, rules out bubbling.

\begin{proof}[Proof of Theorem \ref{thm1}]
   Let $u$ be a solution to \eqref{eq generalflow}. This solution is unique in virtue of the results in \cite{feireisl1997threshold}. Similarly, from \cite{feireisl1997threshold} we  have that $\lim_{t\to \infty} J(u(t)) \geq 0$ with $\lim_{t\to \infty} J(u(t)) = 0$ if and only if $\lim_{t\to \infty} \Vert u(t)\Vert_{W^{1,2}(\R^n)}=0$. Let us assume $\lim_{t\to \infty} J(u(t)) > 0$ and take  $M$ as in \eqref{eq limbubbles}. In virtue of Lemma \ref{lemma linearestimate}, we have that
\begin{equation}\label{eq linear estimate3}
C	\Vert \partial_t u\Vert^2_{L^2(\R^n)} \geq  J(u) -MJ(\xi),
\end{equation}
for $t \geq T$. By combining the dissipation identity \eqref{eq derentropy} with \eqref{eq linear estimate3} we deduce
\begin{equation}\label{eq diff ineq3}
\frac{d}{dt}\Big( J(u) -MJ(\xi)\Big)=  \Vert \partial_t u\Vert_{L^2(\R^n)}^2  \leq -\frac{1}{C}\Bigg( J(u) -MJ(\xi)\Bigg),
\end{equation}
for $t\geq T$. Solving the differential inequality in \eqref{eq diff ineq3}, we can integrate between $t>T$ and $\infty$ to deduce
\begin{equation}\label{eq convergence rate1}
	J(u(t))-MJ(\xi)\leq Ce^\frac{-t}{C}, \qquad{t\geq T}
\end{equation}
which combined with the dissipation identity implies
\begin{equation}\label{eq convergence ut1}
	\int_{t}^\infty \Vert \partial_t u\Vert_{L^2(\R^n)}^2 dt =J(u(t))-MJ(\xi)\leq Ce^\frac{-t}{C},
\end{equation}
for $t\geq T$. We claim now that \eqref{eq convergence ut1} implies that $u(t)$ converges strongly in $L^2(\R^n)$ as $t\to \infty$ which, in turn, combined with \eqref{eq bubbling} would imply that the flow converges to a unique bubble. To see this, let us notice that for $t_1>t_*>T$ with $t_1-t_*\leq 1$, we can combine the fundamental theorem of calculus with Minkowski's generalized inequality and \eqref{eq convergence ut1} to infer
\begin{eqnarray}\notag
	\Vert u(t) -u(t_*)\Vert_{L^{2}(\R^n)} &=& 	\Big(\int_{\R^n} \Big|\int_{t_*}^{t_1} \partial_t u \Big|^2\Big)^\frac{1}{2} \\\notag
	&\leq& \int_{t_*}^{t_1}\Big(\int_{\R^n} |\partial_t u|^2\Big)^\frac{1}{2}\\\label{eq Cauchy1}
		&\leq& Ce^\frac{-t_*}{C}
\end{eqnarray}
	 Then, given any $s_*>t_*$, we can apply \eqref{eq Cauchy1} along the sequence of times $t_i=t_*+i$ for $i=0,\cdots, k$ with $k=\lfloor s_*-t_*\rfloor$ and with $t_{k+1}=s_*$ so that 
	\begin{eqnarray}\notag
	\Vert u(s_*) -u(t_*)\Vert_{L^{2}(\R^n)} &\leq& \sum_{i=0}^{k} \Vert u(t_{i+1}) -u(t_i)\Vert_{L^{2}(\R^n)}\\\notag
	 &\leq& C\sum_{i=0}^{k}e^\frac{-t_i}{C}\\\notag
	  &\leq& Ce^\frac{-t_*}{C}\sum_{i=0}^{k}\Big( e^\frac{-1}{C}\Big)^i\\\label{eq cauchy21}
	   &\leq& Ce^\frac{-t_*}{C}.
	\end{eqnarray}	
	
	In virtue of the sub-sequential bubbling along the flow, \eqref{eq bubbling}, we can take a sequence of $\{s_k\}_{k\in \N}$ with $s_k \to \infty$ and such that $u(s_k)$ converges pointwise to $\tau_{x^*} [\xi]$ for some $x^* \in \R^n$. So, by combining \eqref{eq cauchy21} with Fatou's lemma yields
    \begin{equation*}
       \Vert \tau_{x^*} [\xi] -u(t_*)\Vert_{L^2(\R^n)} \leq  Ce^\frac{-t_*}{C},
    \end{equation*}
    which combined with standard parabolic regularity theory implies
     \begin{equation}\label{eq rateprof}
       \Vert \tau_{x^*} [\xi] -u(t_*)\Vert_{C^2(\R^n)\cap W^{2,2}(\R^n)} \leq  Ce^\frac{-t_*}{C},
    \end{equation}
    and thus, in particular, that $M=1$ in \eqref{eq convergence rate1} as we wanted to show.	
\end{proof}

\subsection{Quantiative weak interaction estimate}

\begin{lemma}\label{lemma weakinter}
  Let $u$ be as in Theorem \ref{thm1} satisfying \eqref{eq limbubbles} and let $\eta(t) = \sum_{i=1}^M \alpha_i(t)\tau_{x^i(t)}\big[\xi\big]$ be a best matching $M$-bubble. Given $\delta>0$, there exists $T(\delta)$ such that if $t\geq T(\delta)$
\begin{equation}\label{eq simplfied interaction} 
\sum_{1\leq i < j \leq M}\int_{\R^n}\tau_{x^i(t)}\big[\xi\big]\tau_{x^j(t)}\big[\xi\big]\leq \delta \Vert \rho(t)\Vert_{L^2(\R^n)} + C\Vert \partial_tu\Vert_{L^2(\R^n)}.
\end{equation}
\begin{equation}\label{eq quantalpha}
    \sum_{i=1}^M|1-\alpha_i|\leq  \delta \Vert \rho(t)\Vert_{L^2(\R^n)} + C\Vert \partial_tu\Vert_{L^2(\R^n)}.
\end{equation}
\end{lemma}
\begin{proof}
\medskip 

\noindent {\it Preparations:} Our first step consists in linearizing $\partial_t u = \Delta u-f(u)$ around each individual bubble. Observe that the $M$-bubble $\eta$ satisfies
\begin{equation*}\label{eq xi}
	\Delta \eta(t) = \sum_{i=1}^M \alpha_i(t)f(\tau_{x^i(t)}\big[\xi\big]).
\end{equation*}

We now introduce notation that will be used throughout the proof. Let $ X(t)= \{x^1(t), x^2(t),\cdots, x^M(t)\}$ be the set of centers at time $t$. Fix any center $y^1(t) \in X(t)$ with corresponding associated coefficient $\alpha(t)$ and consider $X_1(t) = X(t)\setminus\{y^1(t)\}$,  $\omega_{1}(t)= \sum_{x^i(t) \in X_1(t)} \alpha_i(t)\tau_{x^i(t)}\big[\xi\big]$ and $\xi_1(t) = \tau_{y^1(t)}\big[\xi\big]$. For notational simplicity, we suppress the explicit dependence on $t$ in what follows.\\

Subtracting the equations for $u$ and $\eta$, and adding and subtracting suitable terms, we obtain the identity
\begin{align}
\sum_{x^i \in X_1} \alpha_i f(\tau_{x^i}[\xi])
+ \big[\alpha f(\xi_1) - f(\alpha \xi_1)\big]
- f'(\alpha \xi_1)\omega_1
&=
-2\Delta \rho
+ f'(\alpha \xi_1)\rho
+ \partial_t u
\notag \\
&\quad
+ \big[f(u)-f(\eta)-f'(\eta)\rho\big]
\notag \\
&\quad
+ \big[f(\eta)-f(\alpha\xi_1)-f'(\alpha\xi_1)\omega_1\big]
\notag \\
&\quad
+ \big[f'(\eta)\rho - f'(\alpha\xi_1)\rho\big].
\label{eq interid}
\end{align}
he identity \eqref{eq interid} will be central in the derivation of both \eqref{eq simplfied interaction} and \eqref{eq quantalpha}, although the two estimates require different arguments. Thanks to the $C^{1,\beta}$ regularity of $f$, the nonlinear error terms on the right-hand side of \eqref{eq interid} satisfy
\begin{eqnarray}\label{eq aux1}
	|f(u)-f(\eta)-f'(\eta)\rho|&\leq& C|\rho|^{1+\beta},\\\label{eq aux2}
|f(\eta)-f(\alpha\xi_1)-f( \alpha\xi_1)\omega_{1}|&\leq& C|\omega_{1}|^{1+\beta},\\\label{eq aux3}
	|f'(\eta)\rho-f'( \eta)\rho|&\leq& C|\omega_{1}|^{\beta}|\rho|.
\end{eqnarray}

As a last preparatory ingredient, we construct a family of cutoff functions. Let $\gamma>0$ (to be chosen later) and consider a radial smooth function$0\le \phi\le 1$ satisfying:
\begin{itemize}
\item $\phi=\phi_\gamma\in C_c^\infty(\R^n)$, 
      $\phi\equiv 1$ in $B_{r_1}$, and 
      $|\nabla\phi|\le\gamma$;
\item $r_1=r_1(\gamma)$ is chosen so that
\begin{equation}
\label{eq bound bump}
\xi(x) \le \frac{\gamma}{M}
\quad \text{for } x\in \R^n\setminus B_{r_1}.
\end{equation}
The existence of such $r_1$ follows from the decay of $\xi$ (see Lemma~\ref{lemma pairwise inter});

\item moreover, Lemma~\ref{lemma pairwise inter} allows us to choose $r_1$ so that
\begin{equation}
\label{eq L1 bound bump}
\int_{\R^n\setminus B_{r_1}}
(\xi')^2 + \xi^2
\le \gamma.
\end{equation}
\end{itemize}

\medskip

\noindent {\it Step 2:} We prove \eqref{eq simplfied interaction}.\\

Set $I(t) =  \sum_{1\leq i < j \leq M}\int_{\R^n}\tau_{x^i(t)}\big[\xi\big]\tau_{x^j(t)}\big[\xi\big]$. We aim to control quantitatively the pairwise interactions
\begin{equation}\label{eq pairwise interaction}
	g_{ij}(t) = \int_{\R^n} \tau_{x^j(t)}\big[\xi\big] \tau_{x^i(t)}\big[\xi\big]
\end{equation}
 for $i, j =1\cdots , M$ with $i \neq j$. Before specifying the order in which these interactions will be estimated, we establish a general bound. Fix a center $y^1(t)\in X(t)$ (to be chosen later) and a unit vector $e_1(t)\in\mathbb{S}^{n-1}$, also to be specified. We also suppress the explicit dependence on $t$ of these objects in what follows.

Since $f\in C^{1,\beta}$, we have the expansion near zero
\begin{equation}\label{eq lin at 0}
	\sum_{x^i \in X_1} \alpha_i f(\tau_{x^i}\big[\xi\big]) = f'(0)\omega_{1}+ \sum_{x^i \in X_1}O(\tau_{x^i}\big[\xi\big]^{1+\beta}).
\end{equation}

Combining \eqref{eq interid} with \eqref{eq lin at 0}, and using again the $C^{1,\beta}$ regularity of $f$, we obtain
\begin{align}
(f'(0)-f'(\xi_1))\omega_1
&=
-\Delta\rho
+ f'(\xi_1)\rho
+ \partial_t u
\notag \\
&\quad
+ \big[f(u)-f(\eta)-f'(\eta)\rho\big]
\notag \\
&\quad
+ \big[f(\eta)-f(\alpha\xi_1)-f'(\alpha\xi_1)\omega_1\big]
\notag \\
&\quad
+ \big[f'(\eta)\rho - f'(\xi_1)\rho\big]
\notag \\
&\quad
+ O\big(|1-\alpha|^\beta |\rho|\big)
+ \sum_{x^i\in X_1}
O\big(\tau_{x^i}[\xi]^{1+\beta}
+ |1-\alpha|^\beta \tau_{x^i}[\xi]\xi_1^\beta\big).
\label{eq interation equation}
\end{align}

By combining \eqref{eq interid} and \eqref{eq lin at 0}, and using one more time the $C^{1,\beta}$ regularity of $f$, we have
\begin{eqnarray}\notag
(f'(0) - f'(\xi_1))\omega_1&=&-  \Delta \rho+f'(\xi_1) \rho + \partial_t u  \\\notag
	&&+[f(u)-f(\eta)-f'(\eta)\rho]\\\notag
	&&+ [f(\eta)-f(\alpha\xi_1)-f'( \alpha\xi_1)\omega_1]\\\notag
	&&+ [f'(\eta)\rho-f'(\xi_1)\rho]\\\notag
    &&+ \rho O\big(\xi_1|1-\alpha|\big)^\beta\\\label{eq interation equation}
	&&+ \sum_{x^i \in X_1}O(\tau_{x^i}\big[\xi\big]^{1+\beta}+|1-\alpha|^\beta\tau_{x^i}\big[\xi\big]\xi_1^\beta).
\end{eqnarray}

We now multiply \eqref{eq interation equation} by 
$\partial_{e_1}\xi_1\psi_1$, where 
$\psi_1=\tau_{y^1}\phi$ and $\phi=\phi_\gamma$ is the cutoff introduced earlier. Our aim at this stage is to get a bound for the weak interaction between $\xi_1$ and $\omega_{1}$. For convenience, we denote by $E(t)$ quantities vanishing as $t\to\infty$, and by $E(\gamma)$ quantities vanishing as $\gamma\to0^+$. Testing \eqref{eq interation equation} combining it with \eqref{eq aux1}, \eqref{eq aux2}, and \eqref{eq aux2} yields
\begin{eqnarray}\notag
	\Big|\int_{\R^n}\Big(f'(0) - f'(\xi_1)\Big)\omega_1\partial_{e_1} \xi_1 \psi_1\Big|&\leq& \Big|\int_{\R^n} (-   \Delta \rho+f'(\xi_1) \rho) \partial_{e_1} \xi_1 \psi_1\Big|\\\notag
	&& + \Vert \partial_t u\Vert_{L^2(\R^n)} \Vert \partial_{e_1} \xi_1 \psi_1\Vert_{L^2(\R^n)} \\\notag
	&&+C\int_{\R^n} |\rho|^{1+\beta} |\partial_{e_1} \xi_1 \psi_1|\\\notag
	&&+C\int_{\R^n} |\omega_1|^{\beta}|\rho||\partial_{e_1} \xi_1 \psi_1|\\\notag
	&&+ C\int_{\R^n}(|\omega_1|^{1+\beta}+\omega_1|1-\alpha|^\beta)|\partial_{e_1} \xi_1 \psi_1|\\ \label{eq weak estimate1}
    &&+  C|1-\alpha|^\beta\int_{\R^n} |\rho| |\partial_{e_1} \xi_1 \psi_1|
\end{eqnarray} 

We proceed to rewrite or estimate suitably each one of the terms of the previous inequality. Starting from the left hand side, we can integrate by parts to deduce
\begin{equation}\label{eq ft LHS}
	\int_{\R^n} \omega_1\partial_{e_1} \xi_1 \psi_1 = - \int_{\R^n}  (\partial_{e_1}  \omega_1 )\xi_1 \psi_1+    \omega_1 \xi_1 \partial_{e_1}\psi_1,
\end{equation}
and that
\begin{equation}\label{eq st LHS}
	\int_{\R^n} \omega_{1}\partial_{e_1}(f'(\xi_1))  \psi_1 = - \int_{\R^n} f(\xi_1)(\partial_{e_1}  \omega_{1}  \psi_1+    \omega_{1}  \partial_{e_1}\psi_1).
\end{equation} 
Thus, \eqref{eq ft LHS} and \eqref{eq st LHS} altogether yields
\begin{eqnarray}\notag
&-&\int_{\R^n}\Big(f'(0) - f'(\xi_1)\Big)\omega_{1}\partial_{e_1} \xi_1 \psi_1\\\notag
&=& \int_{\R^n} \xi_1 (\partial_{e_1}  \omega_{1} )\Big( f'(0) -\frac{f(\xi_1)}{\xi_1}\Big)  \psi_1\\\notag
&&+f'(0)\int_{\R^n} \partial_{e_1}\psi_1 \xi_1\omega_{1} \\\label{eq lhs}
&&- \int_{\R^n} f(\xi_1) \partial_{e_1}\psi_1 \xi_1\omega_{1} .
\end{eqnarray}

The first term on the right hand side of \eqref{eq lhs} will play a central role in our estimate, whereas the other two will be treated a errors since, in virtue of \eqref{eq small inter},
\begin{equation}\label{eq inter outside ball}
 \int_{\R^n} |\partial_{e_1}\psi_1 \xi_1\omega_1|\leq C\int_{ \text{supp}(\psi_1)^c} \psi_1 \xi_1\omega_1 \leq E(\gamma)I(t).
\end{equation}

Before continuing with the right hand side of \eqref{eq weak estimate1}, let us notice that by differentiating $\Delta \xi_1 = f(\xi_1)$, we have that
\begin{equation}\label{eq equation derivative}
	-   \Delta (\partial_{e_1} \xi_1)+f'(\xi_1) \partial_{e_1} \xi_1=0.
\end{equation}
Therefore, integrating by parts the first term in the right hand side of \eqref{eq weak estimate1} yields
\begin{eqnarray}\label{eq rho error} 
\Big|	\int_{\R^n} (-  \e \Delta \rho+f'(\xi_1) \rho) \partial_{e_1} \xi_1 \psi_1\Big| &=& \Big| \e\int_{\R^n} \nabla \rho \cdot \nabla \psi \partial_{e_1} \xi_1  \Big|\\\notag
&&\leq  C\Vert \rho\Vert_{H^1(\R^n)} \Vert \nabla \psi  \xi_1' \Vert_{L^2(\R^n)}.
\end{eqnarray}

Thus, by combining  \eqref{eq weak estimate1},  \eqref{eq lhs}, \eqref{eq inter outside ball}, and \eqref{eq rho error} yields
\begin{eqnarray}\notag
\Big| \int_{\R^n} \xi_1 (\partial_{e_1}  \omega_{1} )\Big( f'(0) -\frac{f(\xi_1)}{\xi_1}\Big)  \psi_1\Big| &\leq& (E(\gamma)+E(t))(\Vert \rho\Vert_{H^1(\R^n)} + I(t))\\\label{eq key bound}
&&+C\Vert \partial_t u\Vert_{L^2(\R^n)}.
\end{eqnarray}

The rest of the argument is devoted to exploit \eqref{eq key bound} strategically to conclude the claim of this step. Assumption \eqref{eq KPP} guarantees that
\begin{equation}\label{eq positive W}
	f'(0) -\frac{f(\xi)}{\xi}  	\geq 0, \quad \mbox{in $\R^n$}.
\end{equation}
Moreover, since $\xi$ attains its maximum at the origin, $f(\xi)(0)=\Delta \xi(0)\leq 0$, which combined with $f'(0)>0$, implies
\begin{equation}\label{eq strict positivity assump}
 f'(0) -\frac{f(\xi)}{\xi} \geq \frac{1}{C}, \quad \mbox{ in $ B_{R_0}$},
\end{equation}
for some $R_0>0$. Thus \eqref{eq tail zeta'}, \eqref{eq interaction function}, \eqref{eq positive W}, and  \eqref{eq strict positivity assump} imply that for $|x| \geq 2R_0$
\begin{equation}\label{eq interaction ls}
	\int_{B_\frac{R_0}{2}}\Big(f'(0) -\frac{f(\xi)}{\xi}\Big)\xi \tau_{x}\big[|\xi'|\big] \geq 	\frac{1}{C}\int_{\R^n}\xi \tau_{x}\big[\xi\big] .
\end{equation}

The other main ingredient for the estimate is provided by Lemma \ref{lemma ch}. Our plan is to apply it inductively in order to conclude this step. Let us select $y^1 \in X$ and $e_1 \in \mathbb{S}^{n-1}$ given by Lemma \ref{lemma ch}. In virtue of Lemma \ref{lemma ch} and \eqref{eq tail zeta'}, 
\begin{eqnarray}\notag
	-\partial_{e_1} \omega_1 &=&  \sum_{x^i \in X_{1}} \alpha_i \tau_{x^i}\big[-\xi'\big] \frac{(x-x^i)\cdot e_i}{|x-x^i|}\\\notag	
	 &\geq& \frac{1}{C}\sum_{x^i \in X_{1}} \alpha_i \tau_{x^i}\big[-\xi'\big]\\\label{eq derivative far bubbles}
	 &\geq& \frac{1}{C }\omega_1,
\end{eqnarray}
for $x\in B_{r_1+1}(y^1)\subset \text{supp}(\psi_1)$ and where we have taken $T$ sufficinetly large so that $\alpha_i>\frac{1}{2}$. \\

 Hence, by combining \eqref{eq interaction ls} and \eqref{eq derivative far bubbles} we derive the bound
 \begin{equation}\label{eq bad centers iso}
 	\Big| \int_{\R^n} \xi_1 (\partial_{e_1}  \omega_1 )\Big( f'(0) -\frac{f(\xi_1)}{\xi_1}\Big)  \psi_1\Big| \geq \frac{1}{C }\omega_1 \tau_{y^1}\big[\xi \big]
 \end{equation}
which implies
\begin{eqnarray}\label{eq small ball}
	\int_{\R^n}   \tau_{y^1}\big[\xi \big] \sum_{x^i \in X_1}  \tau_{x^i}\big[\xi \big]  \leq \Big((E(t)+E(\gamma))\Vert \rho\Vert_{H^1(\R^n)} +\Vert \partial_t u\Vert_{L^2(\R^n)}\Big).
\end{eqnarray}

Let us define $X_k = X_{k-1} \setminus \{y^k\}$, and $\xi_k = \tau_{y^k} \xi$ with $X_0=X$. Proceeding inductively, let us assume that
\begin{eqnarray}\label{eq inductive interaction}
	\int_{\R^n} \xi_{l} \tau_{y^1}\big[\xi \big] \sum_{x^i \in X_l}  \tau_{x^i}\big[\xi \big] \leq (E(t)+E(\gamma))(\Vert \rho\Vert_{H^1(\R^n)})+C\Vert \partial_t u\Vert_{L^2(\R^n)},
\end{eqnarray}
for $l=1,\cdots, k$. Take $y^{k+1}  \in X_k$ and $e_{k+1} \in \mathbb{S}^{n-1}$ as in Lemma \ref{lemma ch} and set accordingly $\xi_{k+1} = \tau_{y^{k+1}}[\xi]$  and $\omega_{k+1}= \sum_{x^i \in X_{k}} \alpha_i\tau_{x^i}\big[\xi\big]$. Thus, by combining \eqref{eq key bound}, \eqref{eq derivative far bubbles} applied to $\omega_{k+1}$ and \eqref{eq inductive interaction}, we deduce  
\begin{eqnarray}\label{eq inductive interaction}
	\int_{\R^n} \xi_{k+1}\tau_{y^1}\big[\xi \big] \sum_{x^i \in X_{k+1}}  \tau_{x^i}\big[\xi \big] \leq (E(t)+E(\gamma))(\Vert \rho\Vert_{H^1(\R^n)})+C\Vert \partial_t u\Vert_{L^2(\R^n)},
\end{eqnarray}
for $k=1,\cdots, M-1$. So, by adding up all the estimates \eqref{eq inductive interaction} in $k$, we deduce 
\begin{eqnarray*}
I(t) \leq (E(t)+E(\eta))(\Vert \rho\Vert_{H^1(\R^n)})+C\Vert \partial_t u\Vert_{L^2(\R^n)},
	\end{eqnarray*}
from where \eqref{eq simplfied interaction}  follows by taking $T(\delta)$ large enough and $\gamma$ sufficiently small.

\medskip

\noindent {\it Step 3:} We prove \eqref{eq quantalpha}.\\

Starting from \eqref{eq interid}, we rearrange the terms and exploit the  $C^{1,\beta}$ regularity as in \eqref{eq interation equation} to deduce
\begin{eqnarray}\notag
[\alpha f(\xi_1)- f(\alpha \xi_1)]&=&- 2 \Delta \rho+f'(\xi_1) \rho + \partial_t u  \\\notag
	&&+[f(u)-f(\eta)-f'(\eta)\rho]\\\notag
	&&+ [f(\eta)-f(\alpha\xi_1)-f'( \alpha\xi_1)\omega_1]\\\notag
	&&+ [f'(\eta)\rho-f'(\xi_1)\rho]\\\notag
    &&+ \rho O\big(\xi_1|1-\alpha|\big)^\beta\\ \notag
	&&+ \sum_{x^i \in X_1}O(\tau_{x^i}\big[\xi\big]+|1-\alpha|^\beta\tau_{x^i}\big[\xi\big]\xi_1^\beta)\\ \label{eq alpha estimate}
   && +\sum_{x^i \in X_1}  O(\xi_1^\beta\tau_{x^i}\big[\xi\big]).
\end{eqnarray}
We now multiply \eqref{eq alpha estimate} by $\psi\,\xi_1'$ and integrate over $\R^n$. 
Using the expansion
\[
\alpha f(\xi_1)-f(\alpha \xi_1)
=
(1-\alpha)\big[f(\xi_1)-f'(\xi_1)\xi_1\big]
+
O(|1-\alpha|^{1+\beta}),
\]
together with \eqref{eq aux1} and \eqref{eq aux2}, we obtain
\begin{eqnarray}\notag
|1-\alpha|\Big| \int_{\R^n}[f(\xi_1)-f'(\xi_1)\xi_1] \xi_1' \psi \Big| &\leq & \Big|\int_{\R^n} [- 2 \Delta \rho+f'(\xi_1) \rho] \xi_1' \psi \Big| + C\Vert\partial_t u\Vert_{L^2(\R^n)} \\\notag
	&&+C\int_{\R^n} |\rho|^{1+\beta}  |\xi_1'| \psi_1\\\notag
	&&+C\int_{\R^n} |\omega_{1}|^{\beta}|\rho|| \xi_1'\psi_1|\\\notag
	&&+ C\int_{\R^n}(|\omega_1|^{1+\beta}+\omega_1|1-\alpha|^\beta)| \xi_1' \psi_1|\\ \notag
    &&+  C|1-\alpha|^\beta\int_{\R^n} |\rho| | \xi_1' \psi_1|\\ \label{eq alpha2ndbound} 
	&&+C|1-\alpha|^{1+\beta}+ CI(t).
\end{eqnarray}

Testing $\Delta \xi = f(\xi)$ against $\xi'$ gives 
$$\int_{\R^n}[f(\xi_1)-f'(\xi_1)\xi_1] \xi_1'  = - \int_{\R^n} \nabla \xi_1\cdot \nabla \xi' + f'(\xi_1)\xi_1\xi' =-D^2J(\xi_1)(\xi', \xi)<0$$
where the last inequality follows from Lemma~\ref{thm nondegneracy}. 
Hence, choosing the radius $r_1$ (the radius of the ball where $\phi=1$) sufficiently large, we may ensure
\begin{equation}\label{eq posrhsalpha}
    \int_{\R^n}[f(\xi_1)-f'(\xi_1)\xi_1] \xi_1 \psi = -\kappa<0.
\end{equation}
On the other hand, by exploiting \eqref{eq orthpxi} and integrating by parts we obtain
\begin{eqnarray}\notag
   \Big| \int_{\R^n} [- 2 \Delta \rho+f'(\xi_1) \rho] \xi_1' \psi \Big| &\leq& \Big|\int_{\R^n} [ 2 \nabla \rho \cdot \nabla \xi_1'+f'(\xi_1) \rho\xi'] (1- \psi)\Big|+ 2 \int_{\R^n} \xi_1' |\nabla \rho| |\nabla \psi|\\\label{eq finalpha}
   &&\leq C\Vert \rho\Vert_{W^{1,2}(\R^n)}\Vert \xi' \chi_{B_{r_1}^c}\Vert_{L^2(\R^n)}.
\end{eqnarray}
Then, as a combination of \eqref{eq alpha2ndbound}, \eqref{eq posrhsalpha}, and \eqref{eq finalpha}, we have deduced
\begin{equation}\label{eq penultimateal}
    |1-\alpha| \leq (E(t)+E(\gamma))(\Vert \rho\Vert_{H^1(\R^n)})+C\Vert \partial_t u\Vert_{L^2(\R^n)}+C|1-\alpha|^{1+\beta}+ CI(t).
\end{equation}
Since $|1-\alpha(t)|\to 0$ as $t\to\infty$, the term 
$|1-\alpha|^{1+\beta}$ can be absorbed into the left-hand side of \eqref{eq penultimateal} for $t$ sufficiently large. 
Combining this with \eqref{eq simplfied interaction}, we may choose $T$ sufficiently large and $\gamma$ sufficiently small so that
\begin{equation}\label{eq penultimatea2}
    |1-\alpha| \leq \frac{\delta}{M}\Vert \rho\Vert_{H^1(\R^n)}+C\Vert \partial_t u\Vert_{L^2(\R^n)}
\end{equation}
Finally, since $\alpha$ was an arbitrary coefficient $\alpha_i$, summing \eqref{eq penultimatea2} over $i$ yields \eqref{eq quantalpha}.
\end{proof}

\subsection{Multibubble stability dichotomy:} We start this part generalizeing the stability estimate \eqref{eq nondegJ} for $M$-bubbles.

\begin{lemma}\label{lemma multibybble}
Let $\nu >0$ and let $\eta = \sum_{i=1}^M \alpha_i\tau_{x^i}[\xi]$ be a $M$-bubble such that
\begin{equation}\label{eq nu weak interation}
		\sum_{i=1}^M |1-\alpha_i|+\sum_{1\leq i< j \leq M}\int_{\R^n} \tau_{x^i}\big[\xi\big] \tau_{x^j}\big[\xi\big] \leq \nu.
	\end{equation}
Assume that $h\in H^1(\R^n)$ satisfies the almost orthogonality conditions
\begin{equation}\label{eqn a. orthogonality h1}
	\sum_{i=1}^M \sum_{k=1}^n\Big|	\int_{\R^n} h\partial_k\tau_{x^i}[\xi]\Big| \leq \mu \Vert h \Vert_{H^1(\R^n)}
\end{equation}
and
\begin{equation}\label{eqn a. orthogonality h2} 
	\sum_{i=1}^M\Big|D^2J(\tau_{x^i}[\xi])(h,\tau_{x^i}\big[\xi'\big])\Big| \leq \mu \Vert h \Vert_{H^1(\R^n)}
\end{equation}
Then, there exist $\nu_0, \mu_0>0$ (independent of $h$) such that that if $\nu \in (0, \nu_0)$ and $\mu \in [0, \mu_0)$,  then 
\begin{equation}\label{eq stability Mbubble}
	\int_{\R^n} 2 |\nabla h|^2 +f'(\eta) h^2 \geq \frac{1}{C} \Vert h \Vert_{H^1(\R^n)}^2.
\end{equation}
\end{lemma}
\begin{proof}

\medskip

\noindent {\it Preparations:} Let us start noticing that by an elementary projection argument it suffices to prove the case in which \eqref{eqn a. orthogonality h1} and \eqref{eqn a. orthogonality h2} holds with equality, i.e., 
\begin{equation}\label{eqn orthogonality h1}
	\int_{\R^n} h\partial_k\tau_{x^i}[\xi]=0
\end{equation}
and 
\begin{equation}\label{eqn orthogonality h2}
	D^2J(\tau_{x^i}[\xi])(h,\tau_{x^i}\big[\xi'\big])=0
\end{equation}
for $i=1,\cdots, M$ and $k=1,\cdots, n$. Let us also introduce the notation
\begin{equation}\label{eq bilinear M}
	B(v,w):=	\int_{\R^n} 2 \nabla v\cdot \nabla w +f'(\eta) v w.
\end{equation}

As in the proof of Lemma \ref{lemma weakinter}, given $\gamma>0$ to be determined, consider a radial bump function  $0\leq \phi  \leq 1$ such that  $\phi=\phi_{\gamma} \in C_c^\infty(\R^n)$ with $\phi \equiv 1$ in $B_{r_1}$ and with $|\nabla \phi| \leq \gamma$ and satisfying \eqref{eq bound bump} and  \eqref{eq L1 bound bump}. Set $\phi_i=\tau_{x^i}[\phi]$ and let us decompose $h = \sum_{i =1}^{M+1} h_i$ with $h_i =\phi_i h$ for $i=1, \cdots , M$ and $h_{M+1}= h\psi$, with $\psi=1-\sum_{i=1}^M \phi_i$. Upon shrinking the value of $\nu$ in \eqref{eq nu weak interation}, we can take the supports of $\phi_i$ disjoint which implies, of course, that the supports of the $h_i$'s are also disjoint for $i=1,\cdots, M$. Thanks to this observation, we have that 
\begin{equation}\label{eq decomp bilinear}
	B(h,h)= \sum_{i=1}^{M} B(h_i,h_i)+ 2\sum_{i=1}^{M}B(h_{i},h_{M+1})+B(h_{M+1},h_{M+1}).
\end{equation}

\medskip





\noindent {\it Step 1:} We bound $B(h_{M+1},h_{M+1})$ from below.\\

By homogeneity of \eqref{eq stability Mbubble} we may assume $\Vert  h\Vert_{W^{1,2}(\R^n)}=1$. Focusing on the last term of the right hand side of \eqref{eq decomp bilinear}, we can exploit \eqref{eq bound bump} and that $f'(s) >0$ for $s \in [0, \delta_0]$ to guarantee that $f'(\xi)\geq \frac{1}{C}$ in the support of $h_{M+1}$ and therefore
\begin{equation}\label{eq hM+1}
 B(h_{M+1},h_{M+1})\geq \frac{1}{C} \Vert h_{M+1}\Vert_{W^{1,2}(\R^n)}^2.
\end{equation}

\medskip

\noindent {\it Step 2:} We show that \begin{equation}\label{eq quadlb}
    \sum_{i=1}^{M+1}B(h_{i},h_{i})\geq \frac{1}{C} \sum_{i=1}^{M+1} \Vert h_i\Vert_{L^2(\R^n)}^2 
\end{equation}

\medskip

Let us notice first that trivially for $i=1,\cdots, M$
\begin{equation}\label{eq initial trivial bound}
	B(h_i,h_i) \geq -\frac{1}{C_0}\Vert h_i\Vert_{W^{1,2}(\R^n)}^2.
\end{equation} 
Additionally, let us notice that in order to prove \eqref{eq quadlb}, it suffices to show that for certain constant $\theta>0$, if, then
\begin{equation}\label{eq target bound}
\mbox{if } \quad \Vert h_i\Vert_{W^{1,2}(\R^n)}^2\geq \theta, \quad \mbox{then} \quad	B(h_i,h_i) \geq \frac{1}{C} \Vert h_i\Vert_{W^{1,2}(\R^n)}^2.
\end{equation}
Indeed, if \eqref{eq target bound} holds and we take the set of indices $J \subset \{1,\cdots, M+1\}$ (possibly empty) for which $\Vert h_j\Vert_{L^2(\R^n)}^2 \leq \theta $ is arbitrarily small if $j\in J$, then we would have from \eqref{eq initial trivial bound} and  \eqref{eq target bound} that 
\begin{eqnarray}\notag
	 \sum_{i=1}^{M+1}B(h_{i},h_{i}) &=& \sum_{ j\in J} B(h_j,h_j)+\sum_{ j \in \{1,\cdots, M+1\}\setminus J } B(h_j,h_j)\\ \label{eq intermediate red}
	&\geq& \frac{1}{C}\Big( -\frac{C}{C_0} \sum_{ j\in J}  \Vert h_j\Vert_{W^{1,2}(\R^n)}^2 + \sum_{ j \in \{1,\cdots, M+1\}\setminus J } \Vert h_j\Vert_{W^{1,2}(\R^n)}^2 \Big).
\end{eqnarray}

On the other hand, since $\sum_{i=1}^{M+1} \Vert h_i\Vert_{L^2(\R^n)}^2 = \Vert h\Vert_{L^2(\R^n)}^2=1$, then
\begin{equation*}
  -\frac{C}{C_0} \sum_{ j\in J}  \Vert h_j\Vert_{L^2(\R^n)}^2 + \sum_{ j \in \{1,\cdots, M+1\}\setminus J } \Vert h_j\Vert_{L^2(\R^n)}^2 \geq -\frac{C J \theta }{C_0} +1-\theta J\geq \frac{1}{2},
\end{equation*}
if we pick $\theta$ small enough in terms of $C, C_0$ and $M$. This last inequality combined with \eqref{eq intermediate red} yields \eqref{eq quadlb}.\\

We are left to show \eqref{eq target bound}. Let $i \in \{1,\cdots, M\}$ such that
\begin{equation}\label{eq lower bound hi}
	\Vert h_i\Vert_{W^{1,2}(\R^n)}^2\geq \theta
\end{equation} 
 holds. Let us exploit again \eqref{eq bound bump} together with the $\beta$-H\"older regularity of $f'$ to deduce
\begin{equation}\label{eq potential bound}
 |f'(\eta)-f'(\tau_{x^i}[\xi])| \leq C\gamma^\beta, \qquad \mbox{in supp ({$\phi_i$})}
\end{equation}

Therefore,
\begin{equation}\label{eq second var hi}
	B(h_i,h_i) \geq \int_{\R^n} 2 |\nabla h_i|^2+f'(\tau_{x^i}[\xi]) h^2 -C\gamma^\alpha \int_{\R^n} h_i^2.
\end{equation}

Now, turning to the orthogonality conditions, notice that \eqref{eqn orthogonality h2}, \eqref{eq L1 bound bump}, and \eqref{eq lower bound hi} combined imply
\begin{equation}\label{eq almost ort hi der}
	\Big|\int_{\R^n} h_i \partial_k\tau_{x^i}[\xi]\Big|=\Big|\int_{\R^n} -\phi_i h  \partial_k\tau_{x^i}[\xi]\Big|=	\Big|\int_{\R^n} (1-\phi_i)h \partial_k\tau_{x^i}[\xi]\Big|  \leq	C\gamma \Vert h_i\Vert_{W^{1,2}(\R^n)},
\end{equation}
for $k=1,\cdots, n$. Similarly, we can exploit \eqref{eqn orthogonality h2} and combine it with \eqref{eq L1 bound bump} and \eqref{eq lower bound hi} to deduce
\begin{eqnarray}\label{eq almost ort hi}
 	\Big|D^2J(\tau_{x^i}[\xi])(h_i,\tau_{x^i}\big[\xi'\big])\Big| = \Big|D^2J(\tau_{x^i}[\xi])(h (1-\phi_i),\tau_{x^i}\big[\xi'\big])\Big|\leq C\gamma \Vert h_i\Vert_{W^{1,2}(\R^n)}.
\end{eqnarray}

 Hence, thanks to the stability inequality for one bubble combined with \eqref{eq second var hi}, we deduce the existence of $\gamma>0$ small enough (and therefore $\nu_0$) small enough, such that
$$	B(h_i,h_i) \geq\frac{1}{C} \Vert h_i\Vert_{W^{1,2}(\R^n)}^2,$$
which implies \eqref{eq quadlb}.\\

\medskip

\noindent{\it Step 4:} Conclusion.\\

Since $\Vert h\Vert_{W^{1,2}(\R^n)}=1$, we have from the previous step and \eqref{eq decomp bilinear} that
\begin{eqnarray}\label{eq decomp bilinear2}
	B(h,h)\geq \frac{1}{C}+ 2\sum_{i=1}^{M}B(h_{i},h_{M+1}).
\end{eqnarray}
On the other hand, we notice that since $|\nabla \phi|\leq \gamma$, $\phi_i$ and $\phi_j$ are disjoint for $i\neq j$, and that $f'(\eta) >0$ in the support of $\psi$, we have
\begin{eqnarray}\notag
    B(h_i, h_{M+1}) &\geq& \int_{\R^n} (1-\phi_i) h\nabla h\cdot \nabla \phi_i -  \phi_i h\nabla h\cdot \nabla  \phi_i - h^2|\nabla \phi_i|^2 \\\label{eq lastboundmb}
    &\geq& - \gamma C,
\end{eqnarray}
where we also used that $\Vert h\Vert_{W^{1,2}(\R^n)}=1$. Finally, the result follows from combining \eqref{eq decomp bilinear2} with \eqref{eq lastboundmb}.
\end{proof}

We finish this section by proving the stability dichotomy.
\begin{corollary}\label{cor stabilitydich}
   Let $u$ be as in Theorem \ref{thm1} satisfying \eqref{eq limbubbles} and let $\eta(t) = \sum_{i=1}^M \alpha_i(t)\tau_{x^i(t)}\big[\xi\big]$ be a best matching $M$-bubble. Set $\rho(t)=u-\eta(t)$. Then there exists $T>0$ and $\delta_0>0$ such that if $t\geq T$ then either 
    \begin{equation}\label{eq condstab1}
        \Vert \partial_t u\Vert_{L^2(\R^n)} \geq \delta_0 \Vert \rho(t)\Vert_{H^1(\R^n)}
    \end{equation} 
    or
	\begin{equation}\label{eq stability Mbubble1}
		\int_{\R^n}   |\nabla \rho(t)|^2 + f'(\eta(t)) \rho(t)^2 \geq \frac{1}{C}\Vert \rho(t)\Vert_{H^1(\R^n)}^2.
	\end{equation}
\end{corollary}
\begin{proof}
Assume that \eqref{eq condstab1} does not hold, namely, 
\begin{equation}\label{eq small ut}
       \Vert \partial_t u\Vert_{L^2(\R^n)} \leq \delta_0 \Vert \rho(t)\Vert_{H^1(\R^n)},
\end{equation}
with $\delta_0$ to be determined. In virtue of Lemma \ref{lemma weakinter}, given any $\delta>0$, we can find $T(\delta)>0$ such that if $t\geq T(\delta)$, the almost orthogonality condition \eqref{eqn orthogonality 0} implies
\begin{equation*}
		\Bigg|\int_{\R^n} \rho(t) \partial_j \tau_{x^i}[\xi]\Bigg|\leq \delta \Vert \rho(t)\Vert_{H^1(\R^n)} +C \Vert \partial_t u\Vert_{L^2(\R^n)},
	\end{equation*}
which, in turn, thanks to \eqref{eq small ut} implies
\begin{equation}\label{eqn orthogonalitylp}
		\Bigg|\int_{\R^n} \rho(t) \partial_j \tau_{x^i}[\xi]\Bigg|\leq (\delta+\delta_0 C) \Vert \rho(t)\Vert_{H^1(\R^n)}.
	\end{equation}
 
We may now combine \eqref{eqn orthogonalitylp} with the orthogonality conditions \eqref{eq orthpxi}.  Choosing $T$ sufficiently large so that $\delta + C\delta_0$ is sufficiently small and such that \eqref{eq nu weak interation} holds for $\nu$ small enough, we can apply Lemma~\ref{lemma multibybble}.  This yields \eqref{eq stability Mbubble1} for all $t\geq T$.
\end{proof}

\section{Interaction estimates and related lemmata}\label{sec interest}

The first lemma below exploits \eqref{eq concat0} and is key to show that tail interaction is energetically unfavorable in Lemma \ref{lemma difMbubble}. 

\begin{lemma}\label{lemma tailconcavity}
  Let $h:\R \to \R$ with $h(0)=0$ and let $H(t) = \int_0^t h(s)ds$. Assume that $h$ is concave in $(0,\delta)$ for some $\delta>0$. Then, if  $\{x_i\}_{i=1}^M$ is a collection of non-negative numbers satisfying $ \sum_{i=1}^M x_i <\delta$, then
 \begin{equation}\label{eq convexity at 0}
 	H\Big(\sum_{i=1}^M x_i\Big)- \sum_{i=1}^M H(x_i)-	\sum_{1\leq i < j \leq M} h(x_i) x_j \leq 0.
 \end{equation}
\end{lemma}
\begin{proof}
	Reasoning by induction, let us first assume that $M=2$. Since $h$ is concave and thus Lipschitz, we can apply the fundamental theorem of calculus twice to $H(x_1+x_2)-H(x_1)$ and to $H(x_2) $ and exploiting $H(0)=h(0)=0$, we have
	\begin{eqnarray}\notag
		H (x_1+x_2)-  H(x_1)-H(x_2)- h(x_1) x_2 &=& \int_0^1 [h(x_1+rx_2)-h(x_1)-h(rx_2)]x_2\, dr \\\notag
		&=& \int_0^1\int_0^1 [h'(x_1+rsx_2)-h'(srx_2)]x_2^2 \,t \, ds dr.
	\end{eqnarray}
By concavity, we have $h'(x_1+rsx_2) \leq h'(srx_2)$ a.e. $s, r \in (0,1)$, showing the base case.

For the inductive step, we easily conclude by rewriting the left hand side in \eqref{eq convexity at 0} as
\begin{eqnarray*}
 	H \Big(\sum_{i=1}^M x_i\Big)- \sum_{i=1}^M H(x_i)-	\sum_{1\leq i < j \leq M} h(x_i) x_j \\
 	= 	H \Big(\sum_{i=1}^{M} x_i\Big)-	H (x_1)- 	H \Big(\sum_{i=2}^{M} x_i\Big)-	 h(x_1)\sum_{2\leq j \leq M\sum_{1\leq i < j \leq M-1}} x_j \\
 	+ H \Big(\sum_{i+2}^{M} x_i\Big)-\sum_{i=1}^M H(x_i) -\sum_{2\leq i < j \leq M}h(x_i)h(x_j),
\end{eqnarray*}
and by noticing that the second line is non-positive by the base case, whilst the third line is non-positive in virtue of the inductive hypothesis.	
\end{proof}

\begin{lemma}\label{lemma pairwise inter}[Decay and pairwise interaction of bubbles]
Let $n\geq 2$, $\beta\in (0,1]$, and let $f \in C^{1,\beta}(\R)$ satisfying $f(0)=0$, $f'(0)>0$. Let $\xi$ be a positive solution to
\begin{equation}\label{eq gsap}
\begin{cases}
\Delta \xi = f(\xi) \quad  \R^n, \\
 \xi(x) \to 0, \quad  \text{as $|x| \to \infty$}.
\end{cases}
\end{equation} 
Then, $\xi$ is strictly radially decreasing with respect to some point in $\R^n$. Furthermore, if we assume that $\xi$ is radial with respect to the origin, the following properties hold.
\begin{itemize}	
	\item Let us consider the pairwise interaction function 
	$$g(x) := \int_{\R^n} \tau_{x}[\xi]\xi.$$
	There exists $R_1>0$ such that for any $r_0>0$, there exists $C$ depending on $r_0$, $f$ and $n$ such that
	\begin{equation}\label{eq interaction function}
\frac{1}{C}\xi(x)\leq \int_{B_{r_0}} \tau_{x}\xi(y) \xi(y)dy \leq g(|x|) \leq C \xi(x)\quad \mbox{ for $|x|\geq \max\{2,r_0+R_1\}$.}
	\end{equation}

\item 
Given $\delta>0$, there exists $R(\delta)$ such that if $r>R(\delta)$ 
\begin{equation}\label{eq small inter}
	\int_{ B_r^c \cup B_r(x)^c} \tau_{x}[\xi] \xi +\tau_{x}[\xi] |\xi'| + \tau_{x}[|\xi'|] \xi \leq \delta	\int_{\R^n} \tau_{x}[\xi] \xi .
\end{equation}

\end{itemize}
\end{lemma}
\begin{proof}
	
Let $r_0>0$, by the decaying properties of $\xi$, there exists $C(r_0)$ such that $\xi \geq \frac{1}{C(r_0)}\xi(0)$ in $B_{r_0}$, and let $R_1 >0$ such that $f(\xi)\geq \frac{f'(0)}{2}\xi$ in $B_{R_1}^c$, so that
	\begin{equation}\label{eq lb lapzeta}
		 \Delta \xi \geq \frac{1}{C} \xi \mbox{ in $B_r^c$},
	\end{equation}
    for $\geq R_1$. In particular, \eqref{eq lb lapzeta} implies that $\xi$ is subharmonic in $B_r^c$. Together these properties imply, combined with the mean value inequality for subharmonic functions, that for $|x| \geq r_0+R_1$
	\begin{equation}\label{eq lb interaction}
	g(|x|) \geq \frac{1}{C(r_0)}\xi(0) \int_{B_{r_0}} \tau_{x}\xi \geq \frac{1}{C(r_0)} \xi(|x|),
\end{equation}
proving the lower bound in \eqref{eq interaction function}.\\
	
	Let us proceed to estimate $g(|x|)$. Since $\xi$ is strictly radially deceasing, we have
	\begin{eqnarray}\notag
	 g(|x|) &=& \int_{\{|y| < |x|\}} \xi(y-x) \xi(y)dy + \int_{\{|y| \geq |x|\}}\xi(y-x) \xi(y)dy\\\notag
	 && \leq  \int_{\{|y| < |x|\}} \xi(y-x) \xi(y)dy+\xi(|x|)\int_{\{|y| \geq |x|\}} \tau_{x}\xi(y)dy\\\label{eq bd pairwise1}
	 && \leq \int_{\{|y| < |x|\}} \xi(y-x) \xi(y)dy+C\xi(|x|).
	\end{eqnarray}
	
	Additionally, by  \eqref{eq exp bound GNN} we have that for $|x|>4$
	\begin{eqnarray}\notag
		\int_{\{|y| < 1\}} \xi(y-x) \xi(y)dy &\leq& 	C	\int_{\{|y| < 1\}} \frac{e^{-\frac{1}{m}|x-y|}}{|x-y|^\frac{n-1}{2}}\xi(y)dy\\ \notag
		&\leq& C \frac{e^{-\frac{|x|}{m}}}{|x|^\frac{n-1}{2}}	\int_{\{|y| < 1\}} \frac{e^{\frac{|y|}{m}}\xi(y)}{|1-y/|x||^\frac{n-1}{2}}dy\\\label{eq bound ring0}
		&\leq& C \xi(x).
	\end{eqnarray}
	
	Similarly,
		\begin{eqnarray*}
		\int_{\{ |x|-1< |y|<|x|\}} \xi(y-x) \xi(y)dy\leq \xi(|x|-1)\int_{\{ |x|-1< |y|<|x|\}} \xi(y-x)dy\leq C \xi(|x|).
	\end{eqnarray*}	 
	Combining the previous two estimates with \eqref{eq bd pairwise1} yields 
	\begin{equation}\label{eq 2nd bound g}
		g(|x|)\leq C\xi(|x|)+ \int_{\{1\leq |y| \leq |x|-1\}} \xi(y-x) \xi(y)dy.
	\end{equation}	
We claim that the second term on the right hand side of \eqref{eq 2nd bound g} decays faster than $\frac{\xi(r)}{r^k}$ for any $k\in \N$, which would imply \eqref{eq interaction function}. Aiming to prove this claim, let us adopt the notation $y=(\tilde y, y_n) \in \R^{n-1}\times \R$. Since $g$ is radially symmetric we can assume that $x=|x|e_n$ (where $e_n$ is the $n$-th canonical unit vector in $\R^n$). Let us notice that for $|y|< |x|$ we have the elementary inequality 
	\begin{equation}\label{eq el ineq}
	|y-e_n |x||+|y| \geq \frac{|\tilde y|^2}{2|y|}+|x|
\end{equation}
	Indeed, if we take $a = |y-e_n |x||$ and $b=|y|-|x|$, we have that 
	\begin{eqnarray*}
		(a+b)(a-b)&=& |y-e_n|x||^2- (|y|-|x|)^2\\
		&=& (y_n-|x|)^2-y_n^2+2|y||x|-|x|^2\\		
		&=&2|x|(|y|-y_n)
	\end{eqnarray*} 
	which implies
		\begin{eqnarray}\label{eq elementary}
		|y-e_n |x||+|y|-|x| = \frac{2|x|(|y|-y_n)}{|y-e_n |x||+|x|-|y|} \geq |y|-|y_n|,
	\end{eqnarray} 
    by noticing that $|y-e_n |x||+|x|-|y|\leq 2|x|$. On the other hand, from the fundamental theorem of calculus
	\begin{equation*}
		|y|-|y_n| = \sqrt{|\tilde y|^2+|y_n|^2}-|y_n|^2 = \int_0^1 \frac{|\tilde y|^2}{2\sqrt{t|\tilde y|^2+|y_n|^2}}dt \geq \frac{|\tilde y|^2}{2 |y|}
	\end{equation*}	
	which combined with \eqref{eq elementary} yields \eqref{eq el ineq}. We now combine \eqref{eq el ineq} with \eqref{eq exp bound GNN} and use spherical coordinates $y=r\theta$ to deduce
	\begin{eqnarray}\notag
		\int_{\{1<|y| \leq |x|-1\}} \xi(y-|x|e_n) \xi(y)dy &\leq&  C\int_{\{1<|y| \leq |x|-1\}} \frac{e^{-\frac{1}{m}(||x|e_n-y|+|y|)}}{(|(|x|-|y|)|y|)^\frac{n-1}{2}} dy\\\notag
		&\leq&  	C e^{-|x|/m} \int_{\{1<|y| \leq |x|-1\}} \frac{e^{- \frac{|\tilde y|^2}{2|y|m}}}{(|(|x|-|y|)|y|)^\frac{n-1}{2}} dy\\\notag
			&\leq&  	C e^{-|x|/m} \int_{\mathbb{S}^{n-1}}\int_{1}^{|x|-1} e^{- r\frac{|\tilde \theta|^2}{2m}}\Big(\frac{r}{|x|-r}\Big)^\frac{n-1}{2} drd\mathcal{H}^{n-1}(\theta)\\\notag
		&\leq&	 |x|	C e^{-|x|/m}\int_{\mathbb{S}^{n-1}}\int_{0}^{|x|} e^{-  |x|\frac{s}{1+s}\frac{|\tilde \theta|^2}{2m}}\frac{s^\frac{n-1}{2}}{(1+s)^2} dsd\mathcal{H}^{n-1}(\theta)
	\end{eqnarray}
where in the last inequality we used the change of variables $s = \frac{r}{|x|-r}$ (equivalently $r= |x|\frac{s}{1+s}$). We can bound this last integral as follows
\begin{eqnarray*}
\int_{\mathbb{S}^{n-1}}\int_{0}^{|x|} e^{-  |x|\frac{s}{1+s}\frac{|\tilde \theta|^2}{2m}}\frac{s^\frac{n-1}{2}}{(1+s)^2} dsd\mathcal{H}^{n-1}(\theta)&\leq&  C\int_{\mathbb{S}^{n-1}}\int_{0}^{1} e^{-  |x|s\frac{|\tilde \theta|^2}{2m}} dsd\mathcal{H}^{n-1}(\theta) \\
&&+ \int_{\mathbb{S}^{n-1}} e^{-  |x|\frac{|\tilde \theta|^2}{4m}}d\mathcal{H}^{n-1}(\theta)\int_{1}^{|x|} s^\frac{n-5}{2} ds.
\end{eqnarray*}

We claim that both 
$$g_1(r) = \int_{\mathbb{S}^{n-1}}\int_{0}^{1}  e^{-  rs\frac{|\tilde \theta|^2}{2m}} dsd\mathcal{H}^{n-1}(\theta),$$
 and
 $$g_2(r)= \int_{\mathbb{S}^{n-1}}  e^{-  r\frac{|\tilde \theta|^2}{4m}}d\mathcal{H}^{n-1}(\theta),$$
  decay faster than any polynomial. We will focus on $g_1$ since the proof for $g_2$ is completely analogous. We observe first that
  \begin{equation}\label{eq limit poly}
  	\lim_{r\to \infty} \int_{\mathbb{S}^{n-1}}\int_{0}^{1} e^{-  rs\frac{|\tilde \theta|^2}{2m}} (s|\tilde \theta|^2)^k dsd\mathcal{H}^{n-1}(\theta)=0,
  \end{equation}
  for any $k\in \N\cup \{0\}$. Notice first that if \eqref{eq limit poly} holds for $k=0$, then holds for every $k$ since $(s|\tilde \theta|^2)^k\leq 1$. On the other hand, we have that for any $\delta>0$ 
   \begin{eqnarray*}
  \int_{\mathbb{S}^{n-1}}\int_{0}^{1} e^{-  rs\frac{|\tilde{\theta}|^2}{2m}}  dsd\mathcal{H}^{n-1}(\theta)&=&  \int_{\mathbb{S}^{n-1}\cap \{|\tilde \theta| \leq \delta \}}\int_{0}^{1}  e^{-  rs\frac{|\tilde \theta|^2}{2m}}  dsd\mathcal{H}^{n-1}(\theta)\\
  &&+  \int_{\mathbb{S}^{n-1}\cap \{|\tilde \theta| \geq \delta \}}\int_{0}^{\delta} e^{-  rs\frac{|\tilde \theta|^2}{2m}}  dsd\mathcal{H}^{n-1}(\theta) \\
  &&+ \int_{\mathbb{S}^{n-1}\cap \{|\tilde \theta| \geq \delta \}}\int_{\delta}^{1} e^{-  rs\frac{|\tilde \theta|^2}{2m}}  dsd\mathcal{H}^{n-1}(\theta)\\
  &\leq& C\delta + \int_{\mathbb{S}^{n-1}\cap \{|\tilde \theta| \geq \delta \}}\int_{\delta}^{1} e^{-  r\frac{\delta^3}{2m}}  dsd\mathcal{H}^{n-1}(\theta).
  \end{eqnarray*}
Thus,
\begin{equation*}
 \limsup_{r\to \infty}   \int_{\mathbb{S}^{n-1}}\int_{0}^{1} e^{-  rs\frac{|\tilde{\theta}|^2}{2m}}  dsd\mathcal{H}^{n-1}(\theta)\leq C\delta,
\end{equation*}
implying \eqref{eq limit poly} from the arbitrariness of $\delta$. Thus, L'H\^opital's rule allows to deduce that $\lim_{k\to \infty}g_1(r)r^k \to 0$ from \eqref{eq limit poly}. This finally implies that $g(|x|)\leq C\xi(x)$, i.e., the upper bound in \eqref{eq interaction function}.\\

We complete the proof by showing \eqref{eq small inter}. Thanks to \eqref{eq tail zeta'} it suffices to prove that given $\delta>0$, there exists $R(\delta)$ such that if $r>R(\delta)$ 
\begin{equation}\label{eq small inter2}
	\int_{ B_r^c \cup B_r(x)^c} \tau_{x}[\xi] \xi \leq \delta	\int_{\R^n} \tau_{x}[\xi] \xi .
\end{equation}

Let $A_r(x)=B_r^c \cup B_r(x)^c$. Analogously to the decomposition \eqref{eq bd pairwise1}, we have
	\begin{eqnarray}\label{eq complement pairwise1}
\int_{A_r(x) } \tau_{x}[\xi] \xi \leq  \int_{\{|y| < |x|\}\cap A_r} \xi(y-x) \xi(y)dy+\xi(|x|)\int_{\{|y| \geq |x|\}\cap A_r} \tau_{x}\xi(y)dy.
\end{eqnarray}
We bound the second term on the right of \eqref{eq complement pairwise1} by noticing
\begin{equation}\label{eq bound st}
	\int_{\{|y| \geq |x|\}\cap A_r} \tau_{x}\xi(y)dy\leq \int_{B_r^c} \xi,
\end{equation}
which vanishes as $r\to \infty$. On the other hand, from the previous part of the proof, we have that
\begin{equation}\label{eq fast decay}
	\int_{\{1<|y| < |x|-1\}} \xi(y-x) \xi(y)dy \leq C_k\frac{\xi(|x|)}{|x|^k},
\end{equation}
for any $k\in \N$. So, assuming $r>1$, it suffices to show that given $\delta>0$, there exists $R(\delta)$ such that if $r\geq R(\delta)$
\begin{equation}\label{eq ring bound}
	\int_{\{|x|-1< |y| \leq |x|\}\cap A_r} \tau_{x}\xi(y)\xi(y)dy\leq \delta \xi(|x|).
\end{equation}

Proceeding as in \eqref{eq bound ring0}, we have that
	\begin{eqnarray*}\notag
	\int_{\{|x|-1< |y| \leq |x|\}\cap A_r} \xi(y-x) \xi(y)dy &\leq& 	C\int_{\{|x|-1< |y| \leq |x|\}\cap A_r} \frac{e^{\frac{-|y|}{m}}\xi(y-x)}{|y|^\frac{n-1}{2}}dy\\\notag
	&\leq& 	C \xi(x)\int_{\{|x|-1< |y| \leq |x|\}\cap A_r} \xi(y-x)dy\\
	&\leq& 	C \xi(x)\int_{B_r^c} \xi(y)dy
\end{eqnarray*}
which, thanks again to the vanishing of $\int_{B_r^c} \xi(y)dy$ as $r\to \infty$, concludes the proof.
\end{proof}

\begin{lemma}\label{lemma exbm}	
  Under the hypothesis of Theorem \ref{thm gen}, there exists $T>0$ sufficiently large such that for $t\geq T$
  \begin{enumerate}
      \item \eqref{eq bestmatching} has a solution and $\lim_{t\to \infty} \Gamma(t)=0$,
      \item $M$-bubbles solving \eqref{eq bestmatching} are $\nu(t)$-interacting with $\lim_{t\to \infty}\nu(t)=0$,
      \item the system of equations \eqref{eq definitionalpha} has a unique solution vector $(\alpha_1(t), \cdots, \alpha_M(t))$,
      \item and $\lim_{t \to \infty} \alpha_i(t)=1$ for $i=1,\cdots, M.$
  \end{enumerate}
\end{lemma}
\begin{proof}
Let us start proving (1). Arguing by contradiction, we can find a sequence of times $\{T_k\}_{k\in \N}$ such that \eqref{eq bestmatching} is not attained. By Proposition \ref{eq Jbubbling}, we can extract a subsequence $\{t_k\}_{k\in \N}$ and a sequence of simple $M$-bubbles $\{\theta_k\}_{k\in \N}$ such that $u_k(x)=u(x,t_k)$ satisfies $\lim_{k\to \infty} \Vert u_k -\theta_k\Vert_{L^2(\R^n)} =0 $. Take $k$ large enough such that $\Vert u_k -\theta_k\Vert_{L^2(\R^n)} \leq \frac{1}{2} \Vert \xi\Vert_{L^2(\R^n)}$. We claim that $\Gamma(t_k)$ admits a solution, leading to the desired contradiction. To verify the validity of this claim it suffices to show that if
\begin{equation*}
    \psi_j = \sum_{i=1}^M  \tau_{x^i_j}\big[\xi\big],
\end{equation*}
is a minimizing sequence for $\Gamma(t_k)$, then the centers $\{x_j^k\}$ must remain bounded in $\R^n$ for $j\in \N$ and $k=1,\cdots, M$. If this were not the case, we would have that $\liminf_{j\to \infty} \Vert u_k -\psi_j\Vert_{L^2(\R^n)} \geq \Vert\xi\Vert_{L^2(\R^n)}$ contradicting the fact that $\{\psi_j\}$ is a minimizing sequence for $\Gamma(t_k)$.
Arguing similarly, by contradiction, it follows that that $\lim_{t \to \infty} \Gamma(t)=0$.\\

The proofs of (2), (3), and (4) follow by a very similar argument. We proceed to outline one by one.\\

In the case of (2), if we assume that $\theta(t_k)$ is a minimizer of $\Gamma(t_k)$ which are not $\nu$-interacting for some $\nu>0$ for $t_k$ sufficiently large, up to extracting a subsequence, in virtue of Proposition \ref{eq Jbubbling} we have that $\Vert u_k -\theta(t_k)\Vert_{L^2(\R^n)}$ with the centers of the $M$-bubble $\theta(t_k)$ drifting apart as $k\to \infty$. This, in virtue of Proposition \ref{prop radialuniqueness}, we have that $\theta(t_k)$ is $\nu(t_k)$-interacting with $\lim_{t\to \infty} \nu(t_k)=0$.\\

In the case of (3) or (4), we assume the existence of a sequence of times where $(\alpha_1(t), \cdots, \alpha_M(t))$ either (3) or (4) does not hold. In either case, we invoke Proposition \ref{eq Jbubbling} to guarantee that along such sequence, namely $\{t_k\}_{k\in \N}$, $\lim_{k\to \infty} \Vert u_k -\theta(t_k)\Vert_{W^{2,2}\cap C^2(\R^n)}=0$. Hence, this property combined with Proposition \ref{prop radialuniqueness} implies that the systems of equations \eqref{eq definitionalpha} can be rewritten as
\begin{eqnarray}\label{eq system}   
\alpha_i(t_k) Q_{ii}^k
+\sum_{j\neq i} \alpha_j(t_k) Q_{ij}^k= Q_{ii}^k +E_k \quad \mbox{for $i=1, \cdots, M$},
\end{eqnarray}
with $Q_{ij}^k = D^2J(\tau_{x^i(t_k)}\big[\xi'\big])(\tau_{x^i(t_k)}\big[\xi'\big], \tau_{x^j(t_k)}\big[\xi\big])$ and $E_k \to 0$ as $k\to \infty$. By proposition \ref{thm nondegneracy} $Q_{ii}^k = D^2J(\xi')(\xi, \xi')>0$, whilst $\lim_{k\to \infty} Q_{ij}^k = 0$ when $j\neq i$. Thus, \eqref{eq system} is a perturbation of a linear system whose coefficient matrix is an invertible diagonal matrix, which implies that $\alpha_i(t_k)$ is uniquely determined. From \eqref{eq system}, it is also clear that $\lim_{k\to \infty}\alpha_i(t_k)\to 1$, yielding the desired contradiction.
\end{proof}

Our last auxiliary result is a quantitative separation property. Let 
\(P=\{x_1,\dots,x_M\}\subset \mathbb R^n\) be a collection of distinct points, and let \(y\in P\) be an extremal point of \(P\), meaning that \(y\) cannot be written as a non-trivial convex combination of the elements of \(P\setminus\{y\}\). Consider the set of directions
\[
V_y
=
\left\{
\frac{x_i-y}{|x_i-y|}
\,\Bigg|\,
x_i\in P\setminus\{y\}
\right\}
\subset \mathbb S^{n-1}.
\]
Since \(y\) is extremal for \(P\), the set \(V_y\) is strictly contained in a hemisphere. Indeed, otherwise \(0\) would belong to the convex hull of \(V_y\). Hence there would exist coefficients \(\lambda_i\ge 0\), with \(\sum_{x_i\ne y}\lambda_i=1\), such that
\[
\sum_{x_i\neq y}
\lambda_i
\frac{x_i-y}{|x_i-y|}
=0.
\]
Rearranging terms gives
\[
\frac{1}{S}
\sum_{x_i\neq y}
\frac{\lambda_i}{|x_i-y|}
x_i
=
y,
\qquad
S
=
\sum_{x_i\neq y}
\frac{\lambda_i}{|x_i-y|},
\]
which expresses \(y\) as a non-trivial convex combination of the remaining points of \(P\), contradicting its extremality.

The following lemma provides a quantitative strengthening of this geometric observation.

\begin{lemma}\label{lemma ch}
Let \(P=\{x_1,\dots,x_M\}\subset \mathbb{R}^n\) be a collection of distinct points. 
Then there exists a constant \(D=D(M,n)>0\) and a point \(y\in P\) for which one can find a unit vector \(e\in \mathbb{S}^{n-1}\) satisfying
\begin{equation}\label{eq hyper bound}
(x_i-y)\cdot e 
\ge 
\frac{1}{D(M,n)}\,|x_i-y|,
\qquad 
\text{for every } x_i\in P\setminus\{y\}.
\end{equation}
Moreover, if the points of \(P\) satisfy the separation condition
\[
|x_i-x_j|\ge L 
\qquad \text{for all } i\ne j,
\]
then there exists a constant \(D_2=D_2(M,n)>0\) such that
\begin{equation}\label{eq hyper bound neigh}
(x_i-z)\cdot e 
\ge 
\frac{1}{D_2(M,n)}\,|x_i-z|,
\qquad 
\text{for every } x_i\in P\setminus\{y\},
\end{equation}
whenever \(z\in B_{L'}(y)\), where \(L'=\frac{L}{2D(M,n)}\).
\end{lemma}

\begin{proof}
We first prove \eqref{eq hyper bound} by induction on \(M\). If \(M=2\), the statement is immediate: taking
\[
e=\frac{x_2-x_1}{|x_2-x_1|},
\qquad
y=x_1,
\]
the inequality holds with \(D=1\).

Assume now that the statement holds for any collection of \(M-1\) points.  
Apply the inductive hypothesis to \(P\setminus\{x_M\}\). Then there exist 
\(y_0\in P\setminus\{x_M\}\) and \(e_0\in\mathbb S^{n-1}\) such that
\begin{equation}\label{eq hyper bound0}
(x-y_0)\cdot e_0
\ge
\frac{1}{D(M-1,n)}|x-y_0|,
\qquad
\text{for all } x\in P\setminus\{x_M,y_0\}.
\end{equation}

Set
\[
K=\max\{4D(M-1,n),2\}.
\]
We distinguish two cases.

\medskip
\noindent
\textit{Case 1:} Assume
\begin{equation}\label{eq noortho}
\frac{|(x_M-y_0)\cdot e_0|}{|x_M-y_0|}
\ge
\frac{1}{K}.
\end{equation}

If \((x_M-y_0)\cdot e_0>0\), then \eqref{eq hyper bound} holds with 
\(y=y_0\), \(e=e_0\), and \(D(M,n)=K\).

If \((x_M-y_0)\cdot e_0<0\), we set \(y=x_M\) and \(e=e_0\).  
From \eqref{eq noortho} we obtain
\[
(y_0-x_M)\cdot e_0
\ge
\frac{1}{K}|x_M-y_0|.
\]
For \(x\in P\setminus\{x_M,y_0\}\), we deduce from \eqref{eq hyper bound0} that
\[
(x-x_M)\cdot e_0
=
(x-y_0)\cdot e_0+(y_0-x_M)\cdot e_0
\ge
\frac{1}{D(M-1,n)}|x-y_0|
+
\frac{1}{K}|x_M-y_0|.
\]
Using the triangle inequality and the definition of \(K\), we deduce
\[
(x-x_M)\cdot e_0
\ge
\frac{1}{K}|x-x_M|.
\]

\medskip
\noindent
\textit{Case 2:} Assume
\begin{equation}\label{eq aortho}
\frac{|(x_M-y_0)\cdot e_0|}{|x_M-y_0|}
<
\frac{1}{K}.
\end{equation}
Let
\[
e_1=\frac{x_M-y_0}{|x_M-y_0|},
\qquad
\alpha=\frac{2}{K},
\qquad
e=\frac{e_0+\alpha e_1}{|e_0+\alpha e_1|}.
\]
Then \(|e_0\cdot e_1|\le 1/K\), and therefore
\begin{equation}\label{eq e1e2}
|e_0+\alpha e_1|^2
\le
1+\frac{8}{K^2}
\le 4.
\end{equation}

A direct computation gives
\[
(x_M-y_0)\cdot e
=
|x_M-y_0|
\frac{e_0\cdot e_1+\alpha}{|e_0+\alpha e_1|}
\ge
\frac{1}{2K}|x_M-y_0|.
\]

If \(x\in P\setminus\{x_M,y_0\}\), writing 
\(\nu_x=\frac{x-y_0}{|x-y_0|}\), the inductive hypothesis gives
\(\nu_x\cdot e_0\ge 1/D(M-1,n)\). Hence
\[
(x-y_0)\cdot e
=
|x-y_0|
\frac{e_0\cdot\nu_x+\alpha e_1\cdot\nu_x}{|e_0+\alpha e_1|}
\ge |x-y_0|
\frac{1/D(M-1,n)-\alpha }{2} \ge
\frac{1}{4D(M-1,n)}|x-y_0|.
\]

This completes the proof of \eqref{eq hyper bound}.

\medskip

We now prove \eqref{eq hyper bound neigh}.  
Let \(z\in B_{L'}(y)\) and write \(z=y+z_0\) with \(|z_0|\le L'\).  
For \(x\in P\setminus\{y\}\),
\[
(x-z)\cdot e
=
(x-y)\cdot e - z_0\cdot e
\ge
\frac{1}{D(M,n)}|x-y|-|z_0|.
\]
Since \(|x-y|\ge L\) and 
\(L'=\frac{L}{2D(M,n)}\), we obtain
\[
(x-z)\cdot e
\ge
\frac{1}{2D(M,n)}|x-y|.
\]
Finally,
\[
|x-z|
\le
|x-y|+|z-y|
\le
|x-y|+L'
\le
\Bigl(1+\frac{1}{2D(M,n)}\Bigr)|x-y|,
\]
which yields \eqref{eq hyper bound neigh} with a suitable constant 
\(D_2(M,n)>0\).
\end{proof}

\bibliographystyle{alpha}
\bibliography{references}

@article{aryan2023stability,
  title={Stability of Hardy--Littlewood--Sobolev inequality under bubbling},
  author={Aryan, Shrey},
  journal={Calculus of Variations and Partial Differential Equations},
  volume={62},
  number={8},
  pages={223},
  year={2023},
  publisher={Springer}
}

@incollection{lions1988positive,
  author       = {Lions, Pierre-Louis},
  title        = {On positive solutions of semilinear elliptic equations in unbounded domains},
  booktitle    = {Nonlinear Diffusion Equations and Their Equilibrium States II},
  editor       = {Ni, W.-M. and Peletier, L. A. and Serrin, J.},
  series       = {Mathematical Sciences Research Institute Publications},
  volume       = {13},
  pages        = {85--122},
  year         = {1988},
  publisher    = {Springer, New York},
  doi          = {10.1007/978-1-4613-9608-6_6}
}

@inproceedings{berestycki1980existence,
  title={Existence of stationary states in nonlinear scalar field equations},
  author={Berestycki, H. and Lions, P.-L.},
  booktitle={Bifurcation Phenomena in Mathematical Physics and Related Topics},
  pages={269--292},
  year={1980},
  publisher={Springer}
}

@article{bonforte2024asymptotic,
  title={Asymptotic behavior of a diffused interface volume-preserving mean curvature flow},
  author={Bonforte, Matteo and Maggi, Francesco and Restrepo, Daniel},
  journal={arXiv preprint arXiv:2407.18868},
  year={2024}
}

@book{bonforte2025stability,
  title={Stability in Gagliardo--Nirenberg--Sobolev inequalities: flows, regularity, and the entropy method},
  author={Bonforte, Matteo and Dolbeault, Jean and Nazaret, Bruno and Simonov, Nikita},
  volume={308},
  number={1554},
  year={2025},
  publisher={American Mathematical Society}
}

@article{busca2002convergence,
  title={Convergence to equilibrium for semilinear parabolic problems in $\mathbb{R}^N$},
  author={Busca, J. and Jendoubi, M. A. and Pol{\'a}{\v{c}}ik, P.},
  journal={Journal of Differential Equations},
  year={2002}
}

@article{ciraolo2018quantitative,
  title={A quantitative analysis of metrics with almost constant positive scalar curvature, with applications to fast diffusion flows},
  author={Ciraolo, Giulio and Figalli, Alessio and Maggi, Francesco},
  journal={International Mathematics Research Notices},
  volume={2018},
  number={21},
  pages={6780--6797},
  year={2018}
}

@article{cortazar1999uniqueness,
  title={The problem of uniqueness of the limit in a semilinear heat equation},
  author={Cort{\'a}zar, Carmen and del Pino, Manuel and Elgueta, Manuel},
  journal={Communications in Partial Differential Equations},
  volume={24},
  number={11--12},
  pages={2147--2172},
  year={1999}
}

@article{chen1991uniqueness,
  title={Uniqueness of the ground state solutions of ${\Delta} u + f(u) = 0$ in $\mathbb{R}^n$, $n \ge 3$},
  author={Chen, Chiun-Chuan and Lin, Chang-Shou},
  journal={Communications in Partial Differential Equations},
  volume={16},
  number={8--9},
  pages={1549--1572},
  year={1991}
}

@article{deng2025sharp,
  title={Sharp quantitative estimates of Struwe's decomposition},
  author={Deng, Bin and Sun, Liming and Wei, Jun-Cheng},
  journal={Duke Mathematical Journal},
  volume={174},
  number={1},
  pages={159--228},
  year={2025}
}

@article{de2023stability,
  title={Stability with explicit constants of the critical points of the fractional Sobolev inequality and applications to fast diffusion},
  author={De Nitti, Nicola and K{\"o}nig, Tobias},
  journal={Journal of Functional Analysis},
  volume={285},
  number={9},
  pages={110093},
  year={2023}
}

@article{feireisl1997threshold,
  title={Convergence to a ground state as a threshold phenomenon in nonlinear parabolic equations},
  author={Feireisl, Eduard and Petzeltov{\'a}, Hana},
  journal={Differential and Integral Equations},
  volume={10},
  number={1},
  pages={181--196},
  year={1997}
}

@article{feireisl1997long,
  title={On the long time behaviour of solutions to nonlinear diffusion equations on $\mathbb{R}^n$},
  author={Feireisl, Eduard},
  journal={Nonlinear Differential Equations and Applications (NoDEA)},
  volume={4},
  number={1},
  pages={43--60},
  year={1997}
}

@article{figalli2020sharp,
  title={On the sharp stability of critical points of the Sobolev inequality},
  author={Figalli, Alessio and Glaudo, Federico},
  journal={Archive for Rational Mechanics and Analysis},
  volume={237},
  number={1},
  pages={201--258},
  year={2020}
}

@article{foldes2011convergence,
  title={Convergence to a steady state for asymptotically autonomous semilinear heat equations on $\mathbb{R}^N$},
  author={F{\"o}ldes, Juraj and Pol{\'a}{\v{c}}ik, Pavol},
  journal={Journal of Differential Equations},
  volume={251},
  number={7},
  pages={1903--1922},
  year={2011}
}

@article{gidas1981symmetry,
  title={Symmetry of positive solutions of nonlinear elliptic equations in $\mathbb{R}^n$},
  author={Gidas, Basilis},
  journal={Advances in Mathematics Supplementary Studies},
  pages={369--402},
  year={1981}
}

@article{maggi2024uniform,
  title={Uniform stability in the Euclidean isoperimetric problem for the Allen--Cahn energy},
  author={Maggi, Francesco and Restrepo, Daniel},
  journal={Analysis \& PDE},
  volume={17},
  number={5},
  pages={1761--1830},
  year={2024}
}

@article{mcleod1993uniqueness,
  title={Uniqueness of positive radial solutions of ${\Delta} u + f(u)=0$ in $\mathbb{R}^n$. II},
  author={McLeod, Kevin},
  journal={Transactions of the American Mathematical Society},
  volume={339},
  number={2},
  pages={495--505},
  year={1993}
}

@article{ni1993locating,
  title={Locating the peaks of least-energy solutions to a semilinear Neumann problem},
  author={Ni, Wei-Ming and Takagi, Izumi},
  journal={Duke Mathematical Journal},
  volume={72},
  number={1},
  pages={247--281},
  year={1993}
}

@article{struwe1984global,
  title={A global compactness result for elliptic boundary value problems involving limiting nonlinearities},
  author={Struwe, Michael},
  journal={Mathematische Zeitschrift},
  volume={187},
  number={4},
  pages={511--517},
  year={1984}
}

\end{document}